\documentclass{article}
\usepackage{amsmath,amsthm,amssymb,graphics}
\newtheorem{theorem}{Theorem}[section]
\newtheorem{corollary}[theorem]{Corollary}
\newtheorem{proposition}[theorem]{Proposition}
\newtheorem{lemma}[theorem]{Lemma}
\theoremstyle{definition}

\newtheorem{example}[theorem]{Example}
\theoremstyle{remark}
\newtheorem{remark}[theorem]{Remark}
\numberwithin{equation}{section}
\DeclareMathOperator{\disc}{disc}
\DeclareMathOperator{\pr}{pr}
\DeclareMathOperator{\Km}{Km}
\DeclareMathOperator{\NS}{NS}

\DeclareMathOperator{\Hom}{Hom}
\DeclareMathOperator{\Pic}{Pic}

\DeclareMathOperator{\Prym}{Prym}

\DeclareMathOperator{\Aut}{Aut}
\DeclareMathOperator{\Ker}{Ker}
\DeclareMathOperator{\Image}{Image}

\newcommand\blfootnote[1]{%
  \begingroup
  \renewcommand\thefootnote{}\footnote{#1}%
  \addtocounter{footnote}{-1}%
  \endgroup
}
\title{Bielliptic curves of genus three and the Torelli problem for certain elliptic surfaces}
\author{Atsushi Ikeda}
\date{}
\begin{document}
\maketitle
\begin{abstract}
 We study the Hodge structure of elliptic surfaces which are canonically defined from bielliptic curves of genus three.
 We prove that the period map for the second cohomology has one dimensional fibers, and the period map for the total cohomology is of degree twelve, and moreover, by adding the information of the Hodge structure of the canonical divisor, we prove a generic Torelli theorem for these elliptic surfaces.
 Finally, we give explicit examples of the pair of non-isomorphic elliptic surfaces which have the same Hodge structure on themselves and the same Hodge structure on their canonical divisors.
\end{abstract}
\section{Introduction}\label{104325_1Dec16}
Let $f:Y\rightarrow{B}$ be an elliptic surface with a section.
The Torelli problem asks if the isomorphism class of the surface $Y$ is determined by the isomorphism class of the Hodge structure $H^{i}(Y,\mathbf{Z})$.
In \cite{Ch2}, Chakiris proved that general simply connected elliptic surface with the geometric genus $p_{g}(Y)\geq2$ is determined by the Hodge structure on $H^{2}(Y,\mathbf{Z})$, and it is called generic Torelli theorem.
But we can not find this kind of results in the case when $Y$ has positive irregularity $q(Y)\geq1$.
In this paper, we consider certain elliptic surfaces with $p_{g}(Y)=q(Y)=1$, which are canonically defined from bielliptic curves of genus $3$.

Let $C$ be a nonsingular projective curve of genus $3$.
Then the symmetric square $C^{(2)}$ has the involution $\kappa:C^{(2)}\rightarrow{C^{(2)}}$ which is defined as the extension of the birational involution given by $q+q'+\kappa(q+q')\in|\Omega_{C}^{1}|$ for $q+q'\in{C^{(2)}}$ with $h^{0}(C,\mathcal{O}_{C}(q+q'))<2$.
In this paper, we consider the case when $C$ has a bielliptic involution $\sigma:C\rightarrow{C}$, whose quotient $E=C/\sigma$ is a nonsingular projective curve of genus $1$.
Then the involution $\sigma^{(2)}:C^{(2)}\rightarrow{C^{(2)}}$ commutes with the involution $\kappa$, and we have several quotient surfaces of $C^{(2)}$;
$$
\begin{array}{cccccccccc}
 &&C^{(2)}&&\\
 &\swarrow&\downarrow&\searrow&\\
 A'=C^{(2)}/(\sigma^{(2)}\circ\kappa)&&X'=C^{(2)}/\kappa&&Y'=C^{(2)}/\sigma^{(2)}\\
 &\searrow&\downarrow&\swarrow&&\searrow\\
 &&Z'=C^{(2)}/\langle\sigma^{(2)},\kappa\rangle&&&&E^{(2)}.\\
\end{array}
$$
The quotient $X'=C^{(2)}/\kappa$ is a projective surface of general type, and it has $28$ ordinary double points, which come from the $28$ bitangent lines of the non-hyperelliptic case, or the $28$ pair of distinct Weierstrass points of the hyperelliptic case.
The quotient $Y'=C^{(2)}/\sigma^{(2)}$ is a projective surface of Kodaira dimension $1$, and it has $6$ ordinary double points, which come from the $6$ pair of distinct fixed points of the involution $\sigma$.
We denote by $Y=Y(C/E)$ the minimal resolution of singularities of $Y'$, which is the main object of this paper.
The quotient $A'=C^{(2)}/(\sigma^{(2)}\circ\kappa)$ is a nonsingular projective surface of Kodaira dimension $0$.
We denote by $A$ the minimal model of $A'$.
Then $A$ is isomorphic to the dual abelian surface of the Prym variety $\Prym{(C/E)}$ of the branched double covering $C\rightarrow{E}$, and the Kummer surface of $A$ is isomorphic to the nonsingular minimal model of the quotient $Z'=C^{(2)}/\langle\sigma^{(2)},\kappa\rangle$.
Since the abelian surface $A$ and its Kummer surface are investigated by Barth in \cite{B}, we can apply their results to the study of the surface $Y$.

The surface $Y$ has the canonical elliptic fibration $f:Y\rightarrow{B}$ by
$Y'\rightarrow{E^{(2)}}\rightarrow{\Pic^{(2)}{(E)}=B}$, and the Hodge structure $H^{1}(Y,\mathbf{Z})\simeq{H^{1}(B,\mathbf{Z})}$ recovers only the information of the base curve $B$.
We have to consider the Hodge structure on $H^{2}(Y,\mathbf{Z})$ for the Torelli theorem.
Since the involution $\kappa$ acts trivially on the space $H^{0}(C^{(2)},\Omega_{C^{(2)}}^{2})$ of holomorphic $2$-forms on $C^{(2)}$, we have the coincidence
$H^{0}(C^{(2)},\Omega_{C^{(2)}}^{2})^{\sigma^{(2)}}=
H^{0}(C^{(2)},\Omega_{C^{(2)}}^{2})^{\sigma^{(2)}\circ\kappa}$
of the invariant subspaces for the involutions $\sigma^{(2)}$ and $\sigma^{(2)}\circ\kappa$.
It implies that the Hodge structures on $H^{2}(Y,\mathbf{Z})$ and on $H^{2}(A,\mathbf{Z})$ are essentially equivalent.
In fact, we have the isomorphism $H^{2}(Y,\mathbf{Q})\simeq{H^{2}(A,\mathbf{Q})\oplus\mathbf{Q}(-1)^{\oplus8}}$ of rational Hodge structures.
We remark that the integral cohomology $H^{2}(A,\mathbf{Z})$ is not a direct summand of $H^{2}(Y,\mathbf{Z})$, hence we need more technical arguments to compare the integral Hodge structures.
The detail is given in Theorem~\ref{104530_30Nov16}.

The Torelli problem for the Prym map asks if the isomorphism class of the double covering $C\rightarrow{E}$ is determined by the isomorphism class of the polarized abelian variety $\Prym{(C/E)}$.
In the case for bielliptic curves, it is studied in \cite{MN}.
In our case, the Torelli theorem does not hold, because the dimension of the moduli space of bielliptic curve of genus $3$ is grater than the dimension of the moduli space of polarized abelian surfaces.
Hence we have a nontrivial deformation $\{C_{t}\rightarrow{E_{t}}\}_{t}$ of the bielliptic curve $C\rightarrow{E}$ with the same Prym variety.
By the above observation in Theorem~\ref{104530_30Nov16}, we have a family $\{Y(C_{t}/E_{t})\}_{t}$ of surfaces with the constant Hodge structure $H^{2}(Y(C_{t}/E_{t}),\mathbf{Z})\simeq{H^{2}(Y(C/E),\mathbf{Z})}$.
Since the bielliptic curve $C_{t}\rightarrow{E_{t}}$ is recovered from the elliptic surface $Y(C_{t}/E_{t})$ by Proposition~\ref{112036_17Oct16}, the family $\{Y(C_{t}/E_{t})\}_{t}$ is a nontrivial deformation of the surface $Y$, and it forms a fiber of the period map for the second cohomology.

Let $F_{\xi}$ be the fiber of the elliptic fibration $f:{Y}\rightarrow{B}$, and let $D_{\xi}$ be the fiber of the composition $C^{(2)}\rightarrow{Y'}\rightarrow{B}$ at $\xi\in{B}$.
Then $D_{\xi}\rightarrow{Y}_{\xi}$ is a bielliptic curve of genus $3$ for general $\xi\in{B}$.
By Pantazis' bigonal construction \cite{B}, \cite{P}, we can show that the Prym variety $\Prym{(D_{\xi}/F_{\xi})}$ is the dual abelian surface of $\Prym{(C/E)}$ (Corollary~\ref{104804_30Nov16}).
Since the Hodge structure on the second cohomology of the abelian surface is isomorphic to the Hodge structure of its dual abelian surface by \cite{Sh}, it gives another family $\{Y(D_{\xi}/F_{\xi})\}_{\xi}$ of elliptic surfaces which satisfies $H^{2}(Y(D_{\xi}/F_{\xi}),\mathbf{Z})\simeq{H^{2}(Y(C/E),\mathbf{Z})}$.
If $\Prym{(C/E)}$ is not a self-dual polarized abelian surface, this family does not contain the original surface $Y=Y(C/E)$, hence the fiber of the period map for $H^{2}(Y,\mathbf{Z})$ has $2$ connected components.
Let $\mathcal{M}$ be the set of isomorphism classes of bielliptic curves of genus $3$, which distinguishes isomorphism classes of corresponding elliptic surfaces, and let $\mathcal{H}$ be the set of isomorphism classes of Hodge structures with additional structures (bilinear forms and some effective classes).
Then we have the period map $\mathcal{M}\rightarrow\mathcal{H}$ by the Hodge structure of $Y$ or the canonical divisor $K_{Y}$.
We remark that the effective canonical divisor $K_{Y}$ of $Y$ is uniquely determined because $p_{g}(Y)=1$.
\renewcommand{\labelenumi}
{$(\arabic{enumi})$}
\begin{theorem}\label{173837_17Jan17}
 \begin{enumerate}
  \item The period map for the Hodge structure $H^{2}(Y,\mathbf{Z})$ has $1$ dimensional fibers, and the general fiber consists of $2$ connected components
	(Theorem~\ref{181010_13Dec16}).
  \item The period map for the mixed Hodge structure $\bigoplus_{i=1}^{2}H^{i}(Y,\mathbf{Z})$ has finite fibers, and the general fiber consists of $12$ points
	(Theorem~\ref{163939_14Jan17}).
  \item The period map for the mixed Hodge structure
	$\bigl(\bigoplus_{i=1}^{2}H^{i}(Y,\mathbf{Z})\bigr)
	\oplus{H^{1}(K_{Y},\mathbf{Z})}$
	is generically injective
	(Theorem~\ref{163320_13Jan17}).
 \end{enumerate}
\end{theorem}

We remark that the elliptic surface $Y$ has a section and the non-constant $j$-function (Proposition~\ref{143201_14Oct16}).
Our Theorem~\ref{173837_17Jan17}~(1) implies that the differential of the period map for $H^{2}(Y,\mathbf{Z})$ is not injective, and it contradicts to the infinitesimal Torelli theorem by \cite{S}.
It seems that there is a gap in the proof of Main Theorem~(A) in \cite{S}.
Therefore, in Proposition~\ref{142139_12Oct16}, we give a direct proof for the failure of the infinitesimal Torelli theorem without using Theorem~\ref{173837_17Jan17}~(1).
A class of surfaces whose period map has a positive dimensional fiber is known
in \cite{Ca}, \cite{Ch1}, \cite{Ky}, \cite{T}, \cite{U1}.
Theorem~\ref{173837_17Jan17}~(3) is motivated by the result on the mixed Torelli theorem for these surfaces in \cite{SSU}.
The global Torelli theorem for the period map in Theorem~\ref{173837_17Jan17}~(3) is not true.
We can find the pair of non-isomorphic elliptic surfaces which have the same Hodge structure, and we give several types of examples which are defined over $\mathbf{Q}$ in Example~\ref{175545_13Jan17}, Example~\ref{165513_18Jan17} and Example~\ref{165437_18Jan17}.

This paper proceed as follows.
In Section~\ref{101650_14Nov16}, we show that the pluri-canonical morphism of $Y$ gives the elliptic fibration $Y\rightarrow{B}$, and we describe its singular fibers.
By using the information of the singular fibers and the canonical fiber, we give the way to reconstruct the original bielliptic curve $C\rightarrow{E}$ from the elliptic surface $Y$.
We also compute the $j$-function of the elliptic fibration.
In Section~\ref{115200_2Dec16}, we review some results in \cite{B} on the Prym varieties for the bielliptic curve $C\rightarrow{E}$ of genus $3$, and we show that the Prym variety of the bielliptic curve $D_{\xi}\rightarrow{F_{\xi}}$ is the dual abelian surface of the Prym variety $\Prym{(C/E)}$.
We compute the explicit equation of the canonical model of the curve $D_{\xi}$, which is used in the proof of Theorem~\ref{173837_17Jan17}~(3), and is useful to find examples for the failure of the global Torelli theorem.
In Section~\ref{221934_29Nov16}, we compute the lattice structure on $H^{2}(Y,\mathbf{Z})$ using a basis of $H^{1}(C,\mathbf{Z})$.
We explain the way to construct the Hodge structure $H^{2}(A,\mathbf{Z})$ from the Hodge structure $H^{2}(Y,\mathbf{Z})$, and also explain the way to construct $H^{2}(Y,\mathbf{Z})$ from $H^{2}(A,\mathbf{Z})$.
In Section~\ref{183635_16Oct16}, first we show that the differential of the period map for $H^{2}(Y,\mathbf{Z})$ is not injective, but it is not used in the proof of the Main theorem (Theorem~\ref{173837_17Jan17}).
We prove the Main theorem by using results in Section~\ref{115200_2Dec16} and Section~\ref{221934_29Nov16}, and we give examples for the failure of the global Torelli theorem.

\section{Construction and singular fibers of elliptic surfaces}\label{101650_14Nov16}
\subsection{Construction}
Let $E$ be a nonsingular projective curve of genus $1$, and let $\pi:C\rightarrow{E}$ be a double covering branched at distinct $4$ points $p_{1},\dots,p_{4}\in{E}$.
Then $C$ is a nonsingular projective curve of genus $3$.
In this paper, the covering $\pi:C\rightarrow{E}$ is called bielliptic curve of genus $3$.
We denote by $\sigma:C\rightarrow{C}$ the covering involution of $\pi$, and denote by $\sigma^{(2)}:C^{(2)}\rightarrow{C^{(2)}}$ the induced involution on its symmetric square $C^{(2)}$.
Then the induced morphism $\pi^{(2)}:C^{(2)}\rightarrow{E^{(2)}}$ on the symmetric squares factors through the quotient $C^{(2)}/\sigma^{(2)}$, and the morphism $\phi:C^{(2)}/\sigma^{(2)}\rightarrow{E^{(2)}}$ is a finite double covering of $E^{(2)}$ branched along $\bigcup_{i=1}^{4}\Gamma_{i}$, where $\Gamma_{i}=\{p_{i}+p\in{E^{(2)}}\mid{p\in{E}}\}$.
Let $\nu:Y\rightarrow{C^{(2)}/\sigma^{(2)}}$ be the minimal resolution, and let $\psi:E^{(2)}\rightarrow\Pic^{(2)}{(E)}$ be the $\mathbf{P}^{1}$-bundle defined by $p+p'\mapsto[\mathcal{O}_{E}(p+p')]$.
Then the composition $f=\psi\circ\phi\circ\nu:Y\rightarrow\Pic^{(2)}{(E)}$ gives a fibration of curves of genus $1$.
We remark that the proper transform of $\phi^{-1}(\Gamma_{i})$ in $Y$ gives a section of the fibration.
Let $\eta\in\Pic^{(2)}{(E)}$ be the isomorphism class of the invertible sheaf with $\eta^{\otimes2}=[\mathcal{O}_{E}(p_{1}+\dots+p_{4})]$ which is determined by
$\pi^{*}\eta=[\Omega_{C}^{1}]$.
First, we prove the following proposition.
\begin{proposition}\label{161646_15Oct16}
 $Y$ is a minimal surface with the numerical invariants $p_{g}(Y)=1$, $q(Y)=1$ and $K_{Y}^{2}=0$, and the effective canonical divisor $K_{Y}$ of $Y$ is the fiber $f^{-1}(\eta)$ at $\eta\in\Pic^{(2)}{(E)}$.
\end{proposition}
For $p\in{E}$, we set $\Gamma_{p}=\{p+p'\in{E^{(2)}}\mid{p'\in{E}}\}$, and we denote by $\Lambda_{p}=\psi^{-1}([\mathcal{O}_{E}(2p)])$ the fiber of $\psi$ at $[\mathcal{O}_{E}(2p)]\in\Pic^{(2)}{(E)}$.
We remark that $\Gamma_{p}$ is a section of the $\mathbf{P}^{1}$-bundle $\psi:E^{(2)}\rightarrow\Pic^{(2)}{(E)}$.
\begin{lemma}\label{161043_15Oct16}
 The canonical sheaf $\Omega_{E^{(2)}}^{2}$ is isomorphic to $\mathcal{O}_{E^{(2)}}(\Lambda_{p}-2\Gamma_{p})$ for any $p\in{E}$.
\end{lemma}
\begin{proof}
 We denote by $\epsilon:E\times{E}\rightarrow{E^{(2)}}$ the natural covering, denote by $\pr_{i}:E\times{E}\rightarrow{E}$ the $i$-th projection,
 and denote by $\Delta_{E}$ the diagonal divisor on $E\times{E}$.
 Then we have an isomorphism
 $$
 \mathcal{O}_{E\times{E}}(\epsilon^{-1}(\Lambda_{p})+\Delta_{E})\vert_{\pi_{i}^{-1}(p')}
 \simeq
 \mathcal{O}_{E\times{E}}(2\pr_{1}^{-1}(p)+2\pr_{2}^{-1}(p))\vert_{\pi_{i}^{-1}(p')}
 $$
 for any $p'\in{E}$ and $i=1,2$.
 By the seesaw theorem \cite{M}, we have
 $$
 \epsilon^{*}\mathcal{O}_{E^{(2)}}(\Lambda_{p})\otimes\mathcal{O}_{E\times{E}}(\Delta_{E})
 \simeq\mathcal{O}_{E\times{E}}(\epsilon^{-1}(\Lambda_{p})+\Delta_{E})
 \simeq
 \mathcal{O}_{E\times{E}}(2\pr_{1}^{-1}(p)+2\pr_{2}^{-1}(p))
 \simeq\epsilon^{*}\mathcal{O}_{E^{(2)}}(2\Gamma_{p}).
 $$
 Hence we have
 $$
 \epsilon^{*}\mathcal{O}_{E^{(2)}}(\Lambda_{p}-2\Gamma_{p})\simeq
 \mathcal{O}_{E\times{E}}(-\Delta_{E})\simeq
 \Omega_{E\times{E}}^{2}(-\Delta_{E})\simeq
 \epsilon^{*}\Omega_{E^{(2)}}^{2}.
 $$
 By the injectivity of $\epsilon^{*}:\Pic{(E^{(2)})}\rightarrow\Pic{(E\times{E})}$, we have $\mathcal{O}_{E^{(2)}}(\Lambda_{p}-2\Gamma_{p})\simeq\Omega_{E^{(2)}}^{2}$.
\end{proof}
Let $\mathcal{G}\in\Pic{(E^{(2)})}$ be the isomorphism class of the invertible sheaf with $\mathcal{G}^{\otimes2}=[\mathcal{O}_{E^{(2)}}(\sum_{i=1}^{4}\Gamma_{i})]$ which determines the finite double covering $\phi:C^{(2)}/\sigma^{(2)}\rightarrow{E^{(2)}}$.
\begin{lemma}\label{160027_15Oct16}
 $\epsilon^{*}\mathcal{G}=\pr_{1}^{*}\eta\otimes\pr_{2}^{*}\eta
 \in\Pic{(E\times{E})}$.
\end{lemma}
\begin{proof}
 The restriction of the covering $\phi:C^{(2)}/\sigma^{(2)}\rightarrow{E^{(2)}}$ to the divisor $E\simeq\Gamma_{p}\hookrightarrow{E^{(2)}}$ is isomorphic to the original covering $\pi:C\rightarrow{E}$ for $p\in{E}\setminus\{p_{1},\dots,p_{4}\}$, hence we have
 $$
 (\epsilon^{*}\mathcal{G})\vert_{\pi_{i}^{-1}(p)}=
 \mathcal{G}\vert_{\Gamma_{p}}=
 \eta=
 (\pr_{1}^{*}\eta\otimes\pr_{2}^{*}\eta)\vert_{\pi_{i}^{-1}(p)}
 $$
 for $p\in{E}\setminus\{p_{1},\dots,p_{4}\}$ and $i=1,2$.
 By the seesaw theorem \cite{M}, we have
 $\epsilon^{*}\mathcal{G}=\pr_{1}^{*}\eta\otimes\pr_{2}^{*}\eta$.
\end{proof}
\begin{proof}[Proof of Proposition~\ref{161646_15Oct16}]
 Since the singularities of the branch divisor $\bigcup_{i=1}^{4}\Gamma_{i}$ of the double covering $\phi:C^{(2)}/\sigma^{(2)}\rightarrow{E^{(2)}}$ are at most nodes, the minimal resolution $\nu:Y\rightarrow{C^{(2)}/\sigma^{(2)}}$ is the canonical resolution in the sense of \cite[Lemma~5]{H}.
 Let $p\in{E}$ be a point with $\eta=[\mathcal{O}_{E}(2p)]$.
 By Lemma~\ref{160027_15Oct16}, we have $\mathcal{G}=[\mathcal{O}_{E^{(2)}}(2\Gamma_{p})]$, and by Lemma~\ref{161043_15Oct16}, we have
 $$
 \Omega_{E^{(2)}}^{2}\otimes\mathcal{G}\simeq
 \Omega_{E^{(2)}}^{2}\otimes
 \mathcal{O}_{E^{(2)}}(2\Gamma_{p})
 \simeq\mathcal{O}_{E^{(2)}}(\Lambda_{p}).
 $$
 By \cite[Lemma~6]{H}, we can compute the numerical invariants
 $$
 \chi(Y,\mathcal{O}_{Y})=\frac{1}{2}(2\Gamma_{p}.\ \Lambda_{p})
 +2\chi(E^{(2)},\mathcal{O}_{E^{(2)}})
 =1,\quad
 K_{Y}^{2}=2(\Lambda_{p}.\ \Lambda_{p})=0.
 $$
 Since
 $$
 \Omega_{Y}^{2}\simeq
 (\phi\circ\nu)^{*}(\Omega_{E^{(2)}}^{2}\otimes\mathcal{G})
 \simeq(\phi\circ\nu)^{*}(\mathcal{O}_{E^{(2)}}(\Lambda_{p}))
 \simeq\mathcal{O}_{Y}(f^{-1}(\eta)),
 $$
 we have
 $$
 q(Y)=1+p_{g}(Y)-\chi(Y,\mathcal{O}_{Y})=
 p_{g}(Y)=h^{0}(Y,\mathcal{O}_{Y}(f^{-1}(\eta)))
 =h^{0}(\Pic^{(2)}(E),\mathcal{O}_{\Pic^{(2)}(E)}(\eta))=1,
 $$
 and the fiber $f^{-1}(\eta)$ is the canonical divisor of $Y$.\par
 Let $D$ be a nonsingular rational curve on $Y$.
 Since $\Pic^{(2)}{(E)}$ is not rational, $D$ is contained in a fiber of $f$.
 Then we have $(K_{Y}.\ D)=0$ and $(D^{2})=-2$, hence $Y$ is a minimal surface.
\end{proof}
\begin{corollary}\label{143245_20Jan17}
 There is an isomorphism of Hodge structures
 $H^{1}(Y,\mathbf{Z})\simeq
 H^{1}(E,\mathbf{Z})$.
\end{corollary}
\begin{proof}
 By Proposition~\ref{161646_15Oct16}, the cokernel of the pull-back
 $f^{*}:H^{1}(E,\mathbf{Z})\simeq
 H^{1}(\Pic^{(2)}{(E)},\mathbf{Z})\rightarrow{H^{1}(Y,\mathbf{Z})}$
 by $f:Y\rightarrow\Pic^{(2)}{(E)}\simeq{E}$
 is finite.
 Since $f$ is connected, it is an isomorphism.
\end{proof}
In the following, we denote by $B=\Pic^{(2)}{(E)}$ the base space of the elliptic surface $f:Y\rightarrow{B=\Pic^{(2)}{(E)}}$, and we call the point $\eta\in{B}$ the canonical point of the covering $\pi:C\rightarrow{E}$.
\begin{corollary}\label{171219_16Oct16}
 The pluri-canonical morphism $\Phi_{|\mathcal{O}_{Y}(mK_{Y})|}:Y\rightarrow\mathbf{P}^{m-1}$ factors through the elliptic surface $f:Y\rightarrow{B}$.
\end{corollary}
\begin{proof}
 By Proposition~\ref{161646_15Oct16}, we have
 $(\Omega_{Y}^{2})^{\otimes{m}}=f^{*}\mathcal{O}_{B}(m\eta)$ and
 $h^{0}(Y,(\Omega_{Y}^{2})^{\otimes{m}})=h^{0}(B,\mathcal{O}_{B}(m\eta))=m$.
 Hence the composition
 $$
 Y\overset{f}{\longrightarrow}B\overset{\Phi_{|\mathcal{O}_{B}(m\eta)|}}{\longrightarrow}\mathbf{P}^{m-1}
 $$
 is defined by the pluri-canonical morphism $\Phi_{|\mathcal{O}_{Y}(mK_{Y})|}$.
\end{proof}
\subsection{Singular fibers}
We explain about fibers of $f:Y\rightarrow{B}$.
If $\xi\in{B}=\Pic^{(2)}{(E)}$ is not the class $[\mathcal{O}_{E}(p_{i}+p_{j})]$ for $1\leq{i}<{j}\leq{4}$, then the fiber $f^{-1}(\xi)$ is the double covering of $\psi^{-1}(\xi)\simeq\mathbf{P}^{1}$ branched at $4$-points $\psi^{-1}(\xi)\cap{\bigcup_{i=1}^{4}\Gamma_{i}}$, hence $f^{-1}(\xi)$ is a nonsingular curve of genus $1$.
The fiber of $f$ at $[\mathcal{O}_{E}(p_{1}+p_{2})]\in{B}=\Pic^{(2)}{(E)}$ is not irreducible, and it is of type $\mathstrut_{1}I_{2}$ if $[\mathcal{O}_{E}(p_{1}+p_{2})]\neq[\mathcal{O}_{E}(p_{3}+p_{4})]$, or of type $\mathstrut_{1}I_{4}$ if $[\mathcal{O}_{E}(p_{1}+p_{2})]=[\mathcal{O}_{E}(p_{3}+p_{4})]$, by the notation of Kodaira \cite{K}.
Let $\Sigma\subset{B}$ be the set of the critical points of $f:Y\rightarrow{B}$.
Then we have $3\leq{\sharp\Sigma}\leq{6}$.
The following Lemma is used for recovering the bielliptic curve $C\rightarrow{E}$ from the elliptic surface $Y$ in Proposition~\ref{112036_17Oct16}.
\begin{lemma}\label{184636_16Oct16}
 The image of $\Sigma$ by the morphism
 $
 \Phi_{|\mathcal{O}_{B}(2\eta)|}:B\rightarrow\mathbf{P}^{1}
 $
 is a set of distinct $3$ points in $\mathbf{P}^{1}$.
\end{lemma}
\begin{proof}
 It follows from
 $$
 [\mathcal{O}_{E}(p_{1}+p_{2})]+[\mathcal{O}_{E}(p_{3}+p_{4})]\sim
 [\mathcal{O}_{E}(p_{1}+p_{3})]+[\mathcal{O}_{E}(p_{2}+p_{4})]\sim
 [\mathcal{O}_{E}(p_{1}+p_{4})]+[\mathcal{O}_{E}(p_{2}+p_{3})]\sim
 2\eta,
 $$
 where $\sim$ denotes the linear equivalence on the curve $B=\Pic^{(2)}{(E)}$.
\end{proof}
\begin{remark}\label{110654_1Dec16}
 Let $\Delta_{h}\subset{C^{(2)}}$ be the curve defined by
 $$
 \Delta_{h}=\{q+q'\in{C^{(2)}}\mid
 h^{0}(C,\mathcal{O}_{C}(q+q'))=2\},
 $$
 which is empty if $C$ is a non-hyperelliptic curve.
 Since $C$ is a nonsingular projective curve of genus $3$, the symmetric square $C^{(2)}$ has a birational involution
 $$
 \kappa:C^{(2)}\setminus\Delta_{h}\longrightarrow{C^{(2)}\setminus\Delta_{h}};\
 q+q'\longmapsto{\kappa(q+q')},
 $$
 where $\kappa(q+q')\in{C^{(2)}}$ is defined as the unique member of the linear system $|\Omega_{C}^{1}(-q-q')|$.
 If $C$ is hyperelliptic, then it extends to the regular involution $\kappa:C^{(2)}\rightarrow{C^{(2)}}$, because $\kappa(q+q')=h(q)+h(q')$ for $q+q'\in{C^{(2)}\setminus\Delta_{h}}$, where $h:C\rightarrow{C}$ denotes the hyperelliptic involution.
 Then the involution $\kappa$ commutes with the involution on the base $B$ defined by the double covering $\Phi_{|\mathcal{O}_{B}(2\eta)|}:B\rightarrow\mathbf{P}^{1}$, hence we have a fibration $Y/\kappa\rightarrow\mathbf{P}^{1}$ of curves of genus $1$ in the commutative diagram
 $$
 \begin{array}{ccc}
  Y&\longrightarrow&Y/\kappa\\
  \downarrow& &\downarrow\\
  B&\underset{\Phi_{|\mathcal{O}_{B}(2\eta)|}}{\longrightarrow}&
   \mathbf{P}^{1}.\\
 \end{array}
 $$
\end{remark}

Let $Y=Y(C/E)$ be the surface constructed from a bielliptic curve $\pi:C\rightarrow{E}$.
\begin{proposition}\label{112036_17Oct16}
 The isomorphism class of the bielliptic curve $\pi:C\rightarrow{E}$ is recovered from the surface $Y$.
\end{proposition}
\begin{proof}
 We set $E'=\Phi_{|\mathcal{O}_{Y}(3K_{Y})|}(Y)\subset\mathbf{P}^{2}$.
 By Corollary~\ref{171219_16Oct16}, $E'$ is isomorphic to $E$, and by Proposition~\ref{161646_15Oct16}, the image of the effective canonical divisor $K_{Y}$ of $Y$ by $\Phi_{|\mathcal{O}_{Y}(3K_{Y})|}$ is a point $t\in{E'}$.
 Let $\Sigma'\subset{E'}$ be the set of critical points of the fibration
 $\Phi_{|\mathcal{O}_{Y}(3K_{Y})|}:Y\rightarrow{E'}$.
 By Corollary~\ref{171219_16Oct16} and Lemma~\ref{184636_16Oct16}, we write as $\Sigma'=\{r_{1},\dots,r_{3},s_{1},\dots,s_{3}\}$ with the condition $\Phi_{|\mathcal{O}_{E'}(2t)|}(r_{i})=\Phi_{|\mathcal{O}_{E'}(2t)|}(s_{i})\in\mathbf{P}^{1}$ for $i=1,2,3$.
 We remark that it is possible that $r_{i}=s_{i}$, if it has a singular fiber of type $\mathstrut_{1}I_{4}$.
 We define the point $r_{4}\in{E}$ by
 $\mathcal{O}_{E'}(s_{1}+s_{2}+s_{3}+r_{4})\simeq\mathcal{O}_{E'}(4t)$.
 Then $r_{1},\dots,r_{4}\in{E'}$ are distinct $4$ points, and we have
 $$
 \mathcal{O}_{E'}(r_{1}+\dots+r_{4})\simeq
 \mathcal{O}_{E'}(t+r_{4})^{\otimes2},
 $$
 because $\mathcal{O}_{E'}(r_{i}+s_{i})\simeq\mathcal{O}_{E'}(2t)$ for $i=1,2,3$.
 Proposition~\ref{112036_17Oct16} is proved by the following Lemma.
\end{proof}
\begin{lemma}
 Let $\pi':C'\rightarrow{E'}$ be the double covering branched at the $4$ points $r_{1},\dots,r_{4}$ which is determined by $[\mathcal{O}_{E'}(t+r_{4})]\in\Pic^{(2)}{(E')}$.
 Then $\pi':C'\rightarrow{E'}$ is isomorphic to the original covering $\pi:C\rightarrow{E}$.
\end{lemma}
\begin{proof}
 By Corollary~\ref{171219_16Oct16}, there is an isomorphism $\varphi:E'\rightarrow\Pic^{(2)}{(E)}$ with $\varphi\circ\Phi_{|\mathcal{O}_{Y}(3K_{Y})|}=f$.
 We have to consider two cases for the proof, which depends on the choice of $r_{i}$ from $\{r_{i},s_{i}\}$.
 First we assume that
 $\varphi(r_{i})=[\mathcal{O}_{E}(p_{i}+p_{4})]$
 for $i=1,2,3$.
 Then we have
 $$
 \varphi(s_{1})\otimes\varphi(s_{2})\otimes\varphi(s_{3})
 =[\mathcal{O}_{E}(2(p_{1}+p_{2}+p_{3}))]
 =\eta^{\otimes4}\otimes[\mathcal{O}_{E}(-2p_{4})]
 =\varphi(t)^{\otimes4}\otimes[\mathcal{O}_{E}(-2p_{4})],
 $$
 hence $\varphi(r_{4})=[\mathcal{O}_{E}(2p_{4})]$.
 Let $\varphi_{4}:E\rightarrow\Pic^{(2)}{(E)}$ be the isomorphism defined by $p\mapsto[\mathcal{O}_{E}(p+p_{4})]$.
 Then the isomorphism $\varphi_{4}^{-1}\circ\varphi:E'\rightarrow{E}$ satisfies $(\varphi_{4}^{-1}\circ\varphi)(r_{i})=p_{i}$ for $i=1,\dots,4$, and
 $$
 \eta=[\mathcal{O}_{E}(\varphi_{4}^{-1}(\eta)+p_{4})]
 =[\mathcal{O}_{E}((\varphi_{4}^{-1}\circ\varphi)(t)+(\varphi_{4}^{-1}\circ\varphi)(r_{4}))]
 =(\varphi_{4}^{-1}\circ\varphi)_{*}[\mathcal{O}_{E'}(t+r_{4})].
 $$
 Next we assume that
 $\varphi(s_{i})=[\mathcal{O}_{E}(p_{i}+p_{4})]$
 for $i=1,2,3$.
 Then we have
 $$
 \varphi(s_{1})\otimes\varphi(s_{2})\otimes\varphi(s_{3})
 =[\mathcal{O}_{E}(p_{1}+p_{2}+p_{3}+3p_{4})]
 =\varphi(t)^{\otimes4}\otimes[\mathcal{O}_{E}(-p_{1}-p_{2}-p_{3}+p_{4})],
 $$
 hence $\varphi(r_{4})=[\mathcal{O}_{E}(p_{1}+p_{2}+p_{3}-p_{4})]$.
 Let $\varphi_{\bar{4}}:{E}\rightarrow\Pic^{(2)}{(E)}$ be the isomorphism defined by $p\mapsto[\mathcal{O}_{E}(p_{1}+p_{2}+p_{3}-p)]$.
 Then the isomorphism $\varphi_{\bar{4}}^{-1}\circ\varphi:E'\rightarrow{E}$ satisfies $(\varphi_{\bar{4}}^{-1}\circ\varphi)(r_{i})=p_{i}$ for $i=1,\dots,4$, and
 $$
 \eta=\eta^{\otimes2}\otimes\eta^{\vee}
 =\eta^{\otimes2}\otimes[\mathcal{O}_{E}(-p_{1}-p_{2}-p_{3}+\varphi_{\bar{4}}^{-1}(\eta))]
 =[\mathcal{O}_{E}(\varphi_{\bar{4}}^{-1}(\eta)+p_{4})]
 =(\varphi_{\bar{4}}^{-1}\circ\varphi)_{*}[\mathcal{O}_{E'}(t+r_{4})],
 $$
 where $\eta^{\vee}$ denotes the class of the dual invertible sheaf of $\eta$.
\end{proof}

\begin{proposition}\label{160302_14Oct16}
 The effective canonical divisor $K_{Y}=f^{-1}(\eta)$ is singular if and only if $C$ is a hyperelliptic curve.
 In this case, the fiber $K_{Y}=f^{-1}(\eta)$ is of type $\mathstrut_{1}I_{4}$.
\end{proposition}
\begin{proof}
 Let $q_{i}\in{C}$ be the ramification point of $\pi:C\rightarrow{E}$ with $\pi(q_{i})=p_{i}$.
 If $C$ is a hyperelliptic curve, then the hyperelliptic involution $h$ commutes with the involution $\sigma$, hence $h$ acts on the set $\{q_{1},\dots,q_{4}\}\subset{C}$.
 By \cite[Lemma~(1.9)]{B}, we have $h(q_{i})\neq{q_{i}}$ for $1\leq{i}\leq{4}$.
 We may assume that $q_{2}=h(q_{1})$ and $q_{4}=h(q_{3})$.
 Then we have
 $\mathcal{O}_{C}(q_{1}+q_{2})\simeq
 \mathcal{O}_{C}(q_{3}+q_{4})$,
 hence
 $\pi^{*}\eta=[\mathcal{O}_{C}(q_{1}+\dots+q_{4})]
 =[\mathcal{O}_{C}(2q_{1}+2q_{2})]
 =\pi^{*}[\mathcal{O}_{E}(p_{1}+p_{2})]$.
 Since $\pi^{*}:\Pic{(E)}\rightarrow\Pic{(C)}$ is injective, we have
 $\eta=[\mathcal{O}_{E}(p_{1}+p_{2})]=[\mathcal{O}_{E}(p_{3}+p_{4})]$,
 and the fiber $K_{Y}=f^{-1}(\eta)$ is singular of type $\mathstrut_{1}I_{4}$.\par
 Conversely, if $K_{Y}=f^{-1}(\eta)$ is singular, then
 $\eta=[\mathcal{O}_{E}(p_{i}+p_{j})]$ for some $1\leq{i}<{j}\leq{4}$.
 We may assume that $\eta=[\mathcal{O}_{E}(p_{1}+p_{2})]$.
 Since
 $[\mathcal{O}_{C}(2q_{1}+2q_{2})]=
 \pi^{*}\eta=
 [\mathcal{O}_{C}(q_{1}+\dots+q_{4})]$,
 we have
 $\mathcal{O}_{C}(q_{1}+q_{2})\simeq
 \mathcal{O}_{C}(q_{3}+q_{4})$.
 It means that the linear system $|\mathcal{O}_{C}(q_{1}+q_{2})|$ gives the morphism of degree $2$ to the projective line.
\end{proof}
\subsection{The $j$-function}
\begin{proposition}\label{143201_14Oct16}
 The $j$-function $j:B\rightarrow\mathbf{P}^{1}$ for the elliptic surface $f:Y\rightarrow{B}$ is a covering of degree $12$, and it is ramified at the canonical point $\eta\in{B}$.
\end{proposition}
The statement that the $j$-function is ramified at $\eta\in{B}$ follows from Remark~\ref{110654_1Dec16}.
For a point $\xi\in{B}\setminus\Sigma$, the fiber $f^{-1}(\xi)$ is the double covering of $\psi^{-1}(\xi)\simeq\mathbf{P}^{1}$ branched at $4$ points $\psi^{-1}(\xi)\cap{\bigcup_{i=1}^{4}\Gamma_{i}}$.
Let
$i_{\xi}:\psi^{-1}(\xi)\rightarrow\mathbf{P}^{1}=\mathbf{C}\cup\{\infty\}$
be the isomorphism
defined by
$$
i_{\xi}(\psi^{-1}(\xi)\cap{\Gamma_{1}})=\infty,\
i_{\xi}(\psi^{-1}(\xi)\cap{\Gamma_{2}})=0,\
i_{\xi}(\psi^{-1}(\xi)\cap{\Gamma_{3}})=1.
$$
We defines the $\lambda$-function by
$$
\lambda:B\setminus\Sigma
\longrightarrow\mathbf{C};\
\xi\longmapsto{i_{\xi}(\psi^{-1}(\xi)\cap{\Gamma_{4}})}.
$$
Then $f^{-1}(\xi)$ is isomorphic to the compactification of
$$
\{(x,y)\in\mathbf{A}^{2}\mid
y^{2}=x(x-1)(x-\lambda(\xi))\}.
$$
\begin{lemma}\label{112302_18Oct16}
 The $\lambda$-function coincides with the morphism
 $
 \Phi_{|\mathcal{O}_{B}(2\eta)|}:B\rightarrow\mathbf{P}^{1},
 $
 by taking the coordinate of $\mathbf{P}^{1}=\mathbf{C}\cup\{\infty\}$ as
 $$
 \Phi_{|\mathcal{O}_{B}(2\eta)|}([\mathcal{O}_{E}(p_{1}+p_{4})])=\infty,\
 \Phi_{|\mathcal{O}_{B}(2\eta)|}([\mathcal{O}_{E}(p_{2}+p_{4})])=0,\
 \Phi_{|\mathcal{O}_{B}(2\eta)|}([\mathcal{O}_{E}(p_{3}+p_{4})])=1.
 $$
\end{lemma}
\begin{proof}
 We fix a point $p_{0}\in{E}$ with $\eta=[\mathcal{O}_{E}(2p_{0})]$.
 Let $E\setminus\{p_{0}\}$ be defined by the equation $y^{2}=g(x)$ in $\mathbf{A}^{2}$, where $g(x)$ is a separated monic cubic polynomial.
 We assume that $p_{1},\dots,p_{4}\in{E\setminus\{p_{0}\}}$, because the case $p_{1}=p_{0}$ is easier.
 We denote the coordinate of the point $p_{i}\in{E}$ by $(x,y)=(a_{i},b_{i})$.
 For a point $p=(a,b)\in{E\setminus\{p_{0}\}}$, we define a rational function by
 $$
 f_{p}:\
 E\longrightarrow\mathbf{P}^{1};\
 (x,y)\longmapsto
 \frac{y+b}{x-a}\in\mathbf{P}^{1}
 =\mathbf{C}\cup\{\infty\},
 $$
 which is the morphism given by the linear system
 $|\mathcal{O}_{E}(p+p_{0})|$.
 Hence we have
 $$
 \lambda(\xi)
 =\frac{f_{p}(p_{4})-f_{p}(p_{2})}{f_{p}(p_{4})-f_{p}(p_{1})}\cdot
 \frac{f_{p}(p_{3})-f_{p}(p_{1})}{f_{p}(p_{3})-f_{p}(p_{2})}
 =\frac{\frac{b_{4}+b}{a_{4}-a}-\frac{b_{2}+b}{a_{2}-a}}
 {\frac{b_{4}+b}{a_{4}-a}-\frac{b_{1}+b}{a_{1}-a}}
 \frac{\frac{b_{3}+b}{a_{3}-a}-\frac{b_{1}+b}{a_{1}-a}}
 {\frac{b_{3}+b}{a_{3}-a}-\frac{b_{2}+b}{a_{2}-a}}
 $$
 for general $\xi=[\mathcal{O}_{E}(p+p_{0})]\in\Pic^{(2)}{(E)}$.
 Let $p_{ij}\in{E}$ be the point with $\mathcal{O}_{E}(p_{1}+p_{j}+p_{ij})\simeq\mathcal{O}_{E}(3p_{0})$.
 We assume that $p_{ij}\in{E\setminus\{p_{0}\}}$, because the case $p_{12}=p_{0}$ is easier.
 We denote the coordinate of the point $p_{ij}\in{E}$ by $(x,y)=(a_{ij},b_{ij})$, and we set
 $r=\frac{b_{4}-b_{1}}{a_{4}-a_{1}}+\frac{b_{3}-b_{2}}{a_{3}-a_{2}}$.
 Then we have
 $$
 \begin{cases}
  r=\frac{b_{4}-b_{1}}{a_{4}-a_{1}}+\frac{b_{3}-b_{2}}{a_{3}-a_{2}}
  =\frac{b_{4}-b_{2}}{a_{4}-a_{2}}+\frac{b_{3}-b_{1}}{a_{3}-a_{1}}
  =\frac{b_{4}-b_{3}}{a_{4}-a_{3}}+\frac{b_{2}-b_{1}}{a_{2}-a_{1}},\\
  r^{2}=\frac{(a_{4}-a_{3})(a_{2}-a_{1})}{a_{24}-a_{14}}
  =\frac{(a_{4}-a_{2})(a_{3}-a_{1})}{a_{34}-a_{14}}
  =\frac{(a_{4}-a_{1})(a_{3}-a_{2})}{a_{34}-a_{24}},
 \end{cases}
 $$
 where these equalities are proved from $\mathcal{O}_{E}(p_{1}+\dots+p_{4})\simeq\mathcal{O}_{E}(4p_{0})$ by the elementary computation.
 We set
 \begin{align*}
  A(x)=\frac{b_{4}-b_{2}}{{a_{4}-a_{2}}}x
  +\frac{a_{4}b_{2}-a_{2}b_{4}}{a_{4}-a_{2}},\
  B(x)=\frac{b_{3}-b_{1}}{{a_{3}-a_{1}}}x
  +\frac{a_{3}b_{1}-a_{1}b_{3}}{a_{3}-a_{1}},\\
  C(x)=\frac{b_{4}-b_{1}}{{a_{4}-a_{1}}}x
  +\frac{a_{4}b_{1}-a_{1}b_{4}}{a_{4}-a_{1}},\
  D(x)=\frac{b_{3}-b_{2}}{{a_{3}-a_{2}}}x
  +\frac{a_{3}b_{2}-a_{2}b_{3}}{a_{3}-a_{2}}.
 \end{align*}
 Then
 $$
 \lambda(\xi)
 =\frac{(a_{4}-a_{2})(a_{3}-a_{1})}{(a_{4}-a_{1})(a_{3}-a_{2})}
 \cdot
 \frac{(A(a)+b)(B(a)+b)}{(C(a)+b)(D(a)+b)}
 =c
 \cdot
 \frac{b^{2}+A(a)B(a)+(A(a)+B(a))b}{b^{2}+C(a)D(a)+(C(a)+D(a))b},
 $$
 where we set $c=\frac{(a_{4}-a_{2})(a_{3}-a_{1})}{(a_{4}-a_{1})(a_{3}-a_{2})}$.
 Since $(a_{24},b_{24})=(a_{13},-b_{13})$, we have
 $$
 A(x)+B(x)=r(x-a_{24})
 $$
 and
 $$
 g(x)=A(x)^{2}+(x-a_{2})(x-a_{4})(x-a_{24})
 =B(x)^{2}+(x-a_{1})(x-a_{3})(x-a_{24}),
 $$
 hence
 \begin{align*}
  &2b^{2}+2A(a)B(a)
  =2g(a)+2A(a)B(a)\\
  =&A(a)^{2}+(a-a_{2})(a-a_{4})(a-a_{24})
  +B(a)^{2}+(a-a_{1})(a-a_{3})(a-a_{24})
  +2A(a)B(a)\\
  =&(A(a)+B(a))^{2}+(a^{2}-(a_{1}+a_{2}+a_{3}+a_{4})a
  +a_{1}a_{3}+a_{2}a_{4})(a-a_{24})\\
  =&r^{2}(a-a_{24})^{2}+(a^{2}-(a_{1}+a_{2}+a_{3}+a_{4})a
  +a_{1}a_{3}+a_{2}a_{4})(a-a_{24}).
 \end{align*}
 Since $(a_{14},b_{14})=(a_{23},-b_{23})$, by the same way, we have
 $$
 C(x)+D(x)=r(x-a_{14})
 $$
 and
 $$
 2b^{2}+2C(a)D(a)=
 r^{2}(a-a_{14})^{2}+(a^{2}-(a_{1}+a_{2}+a_{3}+a_{4})a
 +a_{1}a_{4}+a_{2}a_{3})(a-a_{14}).
 $$
 Hence we have
 $$
 \lambda(\xi)
 =c
 \cdot
 \frac{a-a_{24}}{a-a_{14}}\cdot
 \frac{r^{2}(a-a_{24})+(a^{2}-(a_{1}+a_{2}+a_{3}+a_{4})a
 +a_{1}a_{3}+a_{2}a_{4})+2rb}
 {r^{2}(a-a_{14})+(a^{2}-(a_{1}+a_{2}+a_{3}+a_{4})a
 +a_{1}a_{4}+a_{2}a_{3})+2rb}.
 $$
 Since
 $
 c=\frac{a_{34}-a_{14}}{a_{34}-a_{24}}
 $
 and
 $$
 r^{2}(a-a_{24})+(a^{2}-(a_{1}+a_{2}+a_{3}+a_{4})a
 +a_{1}a_{3}+a_{2}a_{4})
 =r^{2}(a-a_{14})+(a^{2}-(a_{1}+a_{2}+a_{3}+a_{4})a
 +a_{1}a_{4}+a_{2}a_{3}),
 $$
 we have
 $$
 \lambda(\xi)=\frac{a_{34}-a_{14}}{a_{34}-a_{24}}\cdot
 \frac{a-a_{24}}{a-a_{14}}
 =\Phi_{|2\eta|}(\xi).
 $$
\end{proof}
\begin{proof}[Proof of Proposition~\ref{143201_14Oct16}]
 By Lemma~\ref{112302_18Oct16}, the $j$-function
 $
 j(\xi)=2^{8}\frac{(\lambda(\xi)^{2}-\lambda(\xi)+1)^{3}}
 {\lambda(\xi)^{2}(\lambda(\xi)-1)^{2}}
 $
 is of degree $12$, and it is ramified at $\eta\in{B}$.
\end{proof}
\begin{corollary}\label{112709_18Oct16}
 If $C$ is not a hyperelliptic, then the Kodaira-Spencer map
 ${T}_{E}\vert_{\eta}\rightarrow{H^{1}(K_{Y},{T}_{K_{Y}})}$ at the canonical fiber $K_{Y}=f^{-1}(\eta)$
 is zero.
\end{corollary}
\begin{proof}
 By Proposition~\ref{160302_14Oct16}, the canonical divisor $K_{Y}$ is nonsingular, and by Proposition~\ref{143201_14Oct16}, the differential of the $j$-function is zero at $\eta\in{B}$, hence the Kodaira-Spencer map is zero.
\end{proof}
\section{Prym varieties and duality}\label{115200_2Dec16}
\subsection{Prym varieties}
In this section, we review some results on the Prym varieties \cite{M2} for the case when they are defined from bielliptic curves of genus $3$.
In this case, the Prym variety is a $(1,2)$-polarized abelian surface, and it is intensively studied by Barth in \cite{B}.
Particularly, we introduce the duality of bielliptic curves of genus $3$, which is given in \cite{P}, and is called Pantazis' bigonal construction.
We will give an explanation for the duality through the elliptic surface $Y$.\par
Let $\pi:C\rightarrow{E}$ be a bielliptic curve of genus $3$, and let $\sigma:C\rightarrow{C}$ be its bielliptic involution.
The Prym variety $P=\Prym{(C/E)}$ is defined as the image of the homomorphism
$$
J(C)\overset{\mathrm{id}-\sigma^{*}}{\longrightarrow}J(C),
$$
where $J(C)$ denotes the Jacobian variety of $C$.
Since $\pi:C\rightarrow{E}$ has ramification points, the kernel of the norm homomorphism
$$
J(C)\simeq\Pic^{(0)}{(C)}\overset{N}{\longrightarrow}
\Pic^{(0)}{(E)}\simeq{J(E)}
$$
is connected, and it coincides with $P$.
The Prym variety $P$ is a $(1,2)$-polarized abelian surface by the restriction $\Theta_{C}\vert_{P}$ of the theta divisor $\Theta_{C}$ on $J(C)$.
Let $A$ be the abelian surface defined as the cokernel of the homomorphism
$\pi^{*}:J(E)\rightarrow{J(C)}$.
Then $A$ is isomorphic to the dual abelian variety $\Pic^{(0)}{(P)}$
of $P$.
The natural composition $P\hookrightarrow{J(C)}\twoheadrightarrow{A}$ is an isogeny, which coincides with the polarization isogeny
$$
P\longrightarrow{A\simeq\Pic^{(0)}{(P)}};\
x\longmapsto[t_{x}^{*}\mathcal{O}_{P}(\Theta_{C}\vert_{P})\otimes
\mathcal{O}_{P}(-\Theta_{C}\vert_{P})],
$$
where $t_{x}:P\rightarrow{P}$ denotes the translation by $x\in{P}$.
Let $\iota:C\rightarrow{A}$ be the composition $C\hookrightarrow{J(C)}\twoheadrightarrow{A}$, where $C\hookrightarrow{J(C)}$ is the Abel-Jacobi embedding defined by fixing a point of $C$.
By \cite[Proposition~(1.8)]{B}, the morphism $\iota$ is a closed immersion, and it gives a $(1,2)$-polarization on $A$.
By \cite[Duality theorem (1.12)]{B}, the class $[\iota(C)]\in{H^{2}(A,\mathbf{Z})}$ corresponds to
$2[\Theta_{C}\vert_{P}]\in{H^{2}(P,\mathbf{Z})}$ by the pull-back
${H^{2}(A,\mathbf{Z})}\hookrightarrow{H^{2}(P,\mathbf{Z})}$,
which means that the polarizations $[\iota(C)]\in{H^{2}(A,\mathbf{Z})}$ and
$[\Theta_{C}\vert_{P}]\in{H^{2}(P,\mathbf{Z})}$ are dual to each other in the sense of \cite[Section~14.4]{BL}.

\begin{lemma}\label{115601_2Dec16}
 The morphism
 $$
 C^{(2)}\longrightarrow{A};\
 q+q'\longmapsto\iota(q)+\iota(q')
 $$
 is a generically finite double covering, and it induces a birational morphism
 $C^{(2)}/(\sigma^{(2)}\circ\kappa)\rightarrow{A}$,
 where $\kappa$ denotes the involution given in Remark~\ref{110654_1Dec16}.
\end{lemma}
\begin{proof}
 Let $\Delta_{\sigma}\subset{C^{(2)}}$ be the curve defined by
 $$
 \Delta_{\sigma}=\{q+q'\in{C^{(2)}}\mid
 q'=\sigma(q)\}.
 $$
 Then $\Delta_{\sigma}$ and $\Delta_{h}$ are contracted by the morphism
 $C^{(2)}\rightarrow{A}$.
 Since the involution $\sigma^{(2)}\circ\kappa$ does not have isolated fixed points, the quotient $C^{(2)}/(\sigma^{(2)}\circ\kappa)$ is a nonsingular surface, and the image of $\Delta_{\sigma}$ in $C^{(2)}/(\sigma^{(2)}\circ\kappa)$ is a $(-1)$-curve.
 If $\Delta_{h}\neq\emptyset$, then the image of $\Delta_{h}$ in $C^{(2)}/(\sigma^{(2)}\circ\kappa)$ is also a $(-1)$-curve disjoint from the image of $\Delta_{\sigma}$.
 Since
 $$
 \mathcal{O}_{C}(q+q')\otimes\pi^{*}\eta\simeq
 \mathcal{O}_{C}(\sigma^{(2)}\circ\kappa(q+q'))\otimes\pi^{*}\mathcal{O}_{E}(\pi(q)+\pi(q'))
 $$
 for $q+q'\in{C^{(2)}}$, the morphism $C^{(2)}\rightarrow{A}$ factors through the quotient $C^{(2)}/(\sigma^{(2)}\circ\kappa)$.
 For $q+q',r+r'\in{C^{(2)}}\setminus(\Delta_{\sigma}\cup\Delta_{h})$,
 we assume that $\iota(q)+\iota(q')=\iota(r)+\iota(r')$ in $A$.
 Then
 $$
 \mathcal{O}_{C}(q+q'-r-r')=\pi^{*}\mathcal{O}_{E}(\pi(x)-\pi(r))
 $$
 for some $x\in{C}$,
 hence we have
 $\mathcal{O}_{C}(q+q'+\sigma(r))=\mathcal{O}_{C}(x+\sigma(x)+r')$.
 If $h^{0}(C,\mathcal{O}_{C}(q+q'+\sigma(r)))=1$, then $q+q'=r+r'$, because $q+q'\notin\Delta_{\sigma}$.
 If $h^{0}(C,\mathcal{O}_{C}(q+q'+\sigma(r)))=2$, then $h^{1}(C,\Omega_{C}^{1}(-q-q'-\sigma(r))))=1$, hence there is a unique point $y\in{C}$ such that
 $q+q'+\sigma(r)+y\in|\Omega_{C}^{1}|$.
 Since $x+\sigma(x)+r'+y,\ \sigma(x)+x+\sigma(r')+\sigma(y)\in|\Omega_{C}^{1}|$,
 we have $\mathcal{O}_{C}(r'+y)\simeq\mathcal{O}_{C}(\sigma(r')+\sigma(y))$.
 If $y\neq\sigma(r')$, then $\sigma(x)=h(x)$, but it is a contradiction to
 \cite[Lemma~(1.9)]{B}.
 Hence we have $y=\sigma(r')$ and $q+q'+\sigma(r)+\sigma(r')\in|\Omega_{C}^{1}|$.
 Since $q+q'\notin\Delta_{h}$, we have $r+r'=\sigma^{(2)}\circ\kappa(q+q')$.
 Hence $C^{(2)}/(\sigma^{(2)}\circ\kappa)\rightarrow{A}$ should be the blow-down of the above $(-1)$-curves.
\end{proof}
\begin{remark}\label{115406_2Dec16}
 The fixed locus of the involution $\sigma^{(2)}\circ\kappa$ on $C^{(2)}$ is given as a component $D$ of the fiber of
 $\psi\circ\pi^{(2)}:C^{(2)}\rightarrow\Pic^{(2)}{(E)}$
 at the canonical point $\eta\in{B=\Pic^{(2)}{(E)}}$;
 $$
 (\psi\circ\pi^{(2)})^{-1}(\eta)=
 \begin{cases}
  D&(\text{if $C$ is non-hyperelliptic}),\\
  D+\Delta_{h}&(\text{if $C$ is hyperelliptic}).
 \end{cases}
 $$
 We remark that $D$ is the normalization of the ``dual'' $C^{\vee}$ of $C$, which is defined later.
\end{remark}
\begin{remark}\label{165630_3Dec16}
 When we take a ramification point $q_{i}$ of $\pi:C\rightarrow{E}$
 as the base point of the embedding
 $$
 \iota:\
 C\longrightarrow{A};\
 q\longmapsto\mathcal{O}_{C}(q-q_{i})\mod{\pi^{*}\Pic^{(0)}{(E)}},
 $$
 the involution $\sigma^{(2)}:C^{(2)}/(\sigma^{(2)}\circ\kappa)\rightarrow{C^{(2)}/(\sigma^{(2)}\circ\kappa)}$ is compatible with the involution $(-1)_{A}$ on $A$.
 Hence the quotient $C^{(2)}/\langle\sigma^{(2)},\kappa\rangle$ is birational to the Kummer surface $\Km{(A)}$ of $A$.
 In fact, the quotient $Y/\kappa$ has $12$ ordinary double points, and $\Km{(A)}$ is the minimal resolution of $Y/\kappa$;
 $$
 \begin{array}{ccccccc}
  & &Y&\longrightarrow&C^{(2)}/\sigma^{(2)}& & \\
  & &\downarrow& &\downarrow& & \\
  \Km{(A)}&\longrightarrow&Y/\kappa&\longrightarrow&
   C^{(2)}/\langle\sigma^{(2)},\kappa\rangle&\longrightarrow&A/(-1)_{A}.\\
 \end{array}
 $$
 By Remark~\ref{110654_1Dec16}, we have a fibration
 $g_{1}:\Km{(A)}\rightarrow{Y/\kappa}\rightarrow\mathbf{P}^{1}=|\mathcal{O}_{B}(2\eta)|$ of curves of genus $1$.
 If $f:Y\rightarrow{B}$ does not have a singular fiber of type $\mathstrut_{1}I_{4}$, then $g_{1}:\Km{(A)}\rightarrow\mathbf{P}^{1}$ has $6$ singular fibers.
 They contains $3$ singular fibers of type $\mathstrut_{1}I_{2}$ which come from the singular fibers of $f:Y\rightarrow{B}$, and the other $3$ singular fibers are of type $I_{0}^{*}$ which appear at $\xi\in{B}$ with $\mathcal{O}_{B}(2\xi)\simeq\mathcal{O}_{B}(2\eta)$ and $\xi\neq\eta$.
 On the other hand, the linear pencil $|\mathcal{O}_{A}(\iota(C))|$ defines a fibration
 $$
 \Phi:A^{\sim}\longrightarrow\mathbf{P}^{1}=|\mathcal{O}_{A}(\iota(C))|,
 $$
 of curves of genus $3$ by the elimination of the base point of the complete linear system
 $|\mathcal{O}_{A}(\iota(C))|$, where ${A}^{\sim}\rightarrow{A}$ is the blow-up at $4$ points $\iota(q_{1}),\dots,\iota(q_{4})$.
 If the fiber of $C_{t}=\Phi^{-1}(t)$ is not singular, then $C_{t}$ is a bielliptic curve of genus $3$, and the bielliptic involution $\sigma_{t}:C_{t}\rightarrow{C_{t}}$ is defined from the involution $(-1)_A$ by \cite[Proposition~(1.6)]{B}.
 Hence the fibration $\Phi$ factors through the quotient $A^{\sim}/(-1)_{A}\simeq{Y/\kappa}$, and it gives another fibration $g_{2}:\Km{(A)}\rightarrow\mathbf{P}^{1}=|\mathcal{O}_{A}(\iota(C))|$ of curves of genus $1$.
 Let $E_{t}$ be the fiber of $g_{2}$ at $t\in\mathbf{P}^{1}$.
 Then the Prym variety $P(C_{t}/E_{t})$ is isomorphic to $P=P(C/E)$ by \cite[Proposition~(1.10)]{B}.
\end{remark}

Let $\pi:C\rightarrow{E}$ be a bielliptic curve of genus $3$.
We denote by $W\subset\Pic^{(2)}{(C)}$ the image of natural morphism $C^{(2)}\rightarrow\Pic^{(2)}{(C)}$, and denote by $P_{\xi}\subset\Pic^{(2)}{(C)}$ the fiber of the norm homomorphism $N:\Pic^{(2)}{(C)}\rightarrow\Pic^{(2)}{(E)}$ at $\xi\in\Pic^{(2)}{(E)}$.
We set $D_{\xi}=W\cap{P_{\xi}}$.
Since the divisors $W$ and $P_{\xi}$ on $\Pic^{(2)}{(C)}$ is stable under the involution
$$
\Pic^{(2)}{(C)}\longrightarrow\Pic^{(2)}{(C)};\
[\mathcal{L}]\longmapsto[\sigma^{*}\mathcal{L}],
$$
it acts on $D_{\xi}$.
If $D_{\xi}$ is nonsingular, then the quotient morphism
$$
D_{\xi}=W\cap{P_{\xi}}\simeq(\psi\circ\pi^{(2)})^{-1}(\xi)
\longrightarrow{(\psi\circ\phi)}^{-1}(\xi)
$$
is a bielliptic curve of genus $3$, where we recall the notation in Section~\ref{101650_14Nov16};
$$
\begin{array}{ccc}
 C^{(2)}&{\longrightarrow}&C^{(2)}/\sigma^{(2)}\\
 \psi\circ\pi^{(2)}\searrow&&\swarrow{\psi\circ\phi}\\
 &\Pic^{(2)}{(E)}.&\\
\end{array}
$$
We define the dual of $C$ by $C^{\vee}=D_{\eta}$, and denote by
$\pi^{\vee}:C^{\vee}\rightarrow{E^{\vee}}$
the quotient by the involution, where $\eta\in\Pic^{(2)}{(E)}$ is the canonical point of the covering $\pi:C\rightarrow{E}$.
If $C$ is not hyperelliptic, then by Proposition~\ref{160302_14Oct16}, the dual $\pi^{\vee}:{C^{\vee}}\rightarrow{E^{\vee}}$ is a nonsingular bielliptic curve.
If $C$ is hyperelliptic, then $E^{\vee}$ is an irreducible rational curve with one node, and $C^{\vee}$ is an irreducible curve of geometric genus $2$ with one node, which is given by contracting the hyperelliptic locus $\Delta_{h}$ form
$(\psi\circ\pi^{(2)})^{-1}(\eta)=D\cup\Delta_{h}$ in Remark~\ref{115406_2Dec16}.
We remark that $P_{\xi}$ is isomorphic to the Prym variety $P=\Prym{(C/E)}\subset{J(C)}$ by the translation
$$
\Pic^{(2)}{(C)}\supset
P_{\xi}\overset{\simeq}{\longrightarrow}{P}
\subset\Pic^{(0)}{(C)};\
[\mathcal{L}]\longmapsto
[\mathcal{L}\otimes\pi^{*}\mathcal{O}_{E}(-p)],
$$
where $p\in{E}$ is a point with $\xi=[\mathcal{O}_{E}(2p)]$.
Let $D_{p}\subset{P}$ be the image of $D_{\xi}=W\cap{P_{\xi}}$ by the above translation, by abusing the notation.
We fix a point $p_{0}\in{E}$ with $\eta=[\mathcal{O}_{E}(2p_{0})]$.
Then the divisor $C^{\vee}\simeq{D_{p_{0}}}\subset{P}$ defines the $(1,2)$-polarization on the Prym variety $P$.
We recall that the embedding $\iota:C\hookrightarrow{A}$ defines the $(1,2)$-polarization on the dual abelian variety $A\simeq\Pic^{(0)}{(P)}$ of $P$.
This is the reason that we call $C^{\vee}$ the dual of $C$.
In fact, we have $(C^{\vee})^{\vee}\simeq{C}$ for any non-hyperelliptic $C$ by Lemma~\ref{095421_30Nov16}.
We remark that the dual bielliptic curve $\pi^{\vee}:C^{\vee}\rightarrow{E^{\vee}}$ also defines two fibrations $g_{i}:\Km{(P)}\rightarrow{\mathbf{P}^{1}}$ in the same way as $g_{i}:\Km{(A)}\rightarrow\mathbf{P}^{1}$ in Remark~\ref{165630_3Dec16}.
\begin{lemma}\label{103122_30Nov16}
 $D_{p}\subset{P}$ is contained in the linear system $|\mathcal{O}_{P}(D_{p_{0}})|$ for any $p\in{E}$.
\end{lemma}
\begin{proof}
 It is clear that $D_{p}$ and $D_{p_{0}}$ are algebraically equivalent, hence
 $[\mathcal{O}_{P}(D_{p}-D_{p_{0}})]\in\Pic^{(0)}{(P)}$.
 Since $D_{p_{0}}$ is an ample divisor on $P$, the polarization homomorphism
 $$
 P\longrightarrow{A\simeq\Pic^{(0)}{(P)}};\
 x\longmapsto[t_{x}^{*}\mathcal{O}_{P}(D_{p_{0}})\otimes
 \mathcal{O}_{P}(-D_{p_{0}})]
 $$
 is surjective, hence there is a point $x\in{P}$ such that
 $t_{x}^{*}\mathcal{O}_{P}(D_{p_{0}})\simeq\mathcal{O}_{P}(D_{p})$.
 We remark that $D_{p}$ contains $\pi^{*}(J(E)_{2})\subset{J(C)}$ for any $p\in{E}$, where $J(E)_{2}$ denotes the set of $2$-torsion points on $J(E)\simeq\Pic^{(0)}{(E)}$.
 Since $(D_{p}.\ D_{p_{0}})=(D_{p_{0}}^{2})=4$, we have $D_{p}\cap{D_{p_{0}}}=\pi^{*}(J(E)_{2})$ or $D_{p}={D_{p_{0}}}$.
 We consider the case $D_{p}\neq{D_{p_{0}}}$.
 By \cite[(1.5)]{B}, the set $\pi^{*}(J(E)_{2})=\pi^{*}(J(E))\cap{P}$ is the base locus of the linear system $|\mathcal{O}_{P}(D_{p})|$, hence the translation point $x\in{P}$ should be contained in $\pi^{*}(J(E)_{2})$, which is the kernel of the polarization homomorphism.
 It implies
 $\mathcal{O}_{P}(D_{p_{0}})\simeq
 t_{x}^{*}\mathcal{O}_{P}(D_{p_{0}})\simeq
 \mathcal{O}_{P}(D_{p})$.
\end{proof}
\begin{remark}\label{111720_14Dec16}
 Let $\mathcal{D}=W\times_{B}E=\{D_{p}\}_{p\in{E}}$ be the algebraic family of the divisor $D_{p}\subset{P}$;
 $$
 \begin{array}{cccccc}
  &W&\longleftarrow&\mathcal{D}&\hookrightarrow&E\times{P}\\
  &\downarrow&\Box&\downarrow&\swarrow&\\
  B=&\Pic^{(2)}{(E)}&\longleftarrow&E.&&\\
 \end{array}
 $$
 Lemma~\ref{103122_30Nov16} gives the commutative diagram
 $$
 \begin{array}{ccc}
  \mathcal{D}&\longrightarrow&{P}^{\sim}\\
  \downarrow& &\downarrow\\
  E&\longrightarrow&|\mathcal{O}_{P}(D_{p_{0}})|,\\
 \end{array}
 $$
 where ${P}^{\sim}$ denotes the blow-up of $P$ at the base locus $\pi^{*}(J(E)_{2})=\pi^{*}(J(E))\cap{P}$.
 Then the morphism $E\rightarrow|\mathcal{O}_{P}(D_{p_{0}})|\simeq\mathbf{P}^{1}$
 is the double covering defined by the linear system $|\eta|$, and the fibrations on the Kummer surfaces are exchanged through the duality as the following diagram;
 $$
 \begin{array}{ccccccccccc}
  &&A
   &\longleftarrow&W&\longleftarrow&\mathcal{D}&\longrightarrow&P& & \\
  & &\text{\rotatebox{-90}{$\!\!\!\dashrightarrow$}}& &\text{\rotatebox{-90}{$\!\!\!\dashrightarrow$}}&
   &\text{\rotatebox{-90}{$\!\!\!\dashrightarrow$}}& &\text{\rotatebox{-90}{$\!\!\!\dashrightarrow$}}& & \\
  & &\Km{(A)}\qquad&\dashleftarrow&Y&\longleftarrow&Y\times_{B}{E}&
   \dashrightarrow&\qquad\Km{(P)}& & \\
  &\mathstrut^{g_{2}}{\swarrow}&\quad\downarrow{g_{1}}&
   &\downarrow& &\downarrow& &g_{2}\downarrow\quad&{\searrow}^{g_{1}}&\\
  &\mathbf{P}^{1}\qquad&\mathbf{P}^{1}&
   \underset{\Phi_{|\mathcal{O}_{B}(2\eta)|}}{\longleftarrow}&B&
   \underset{1:4}{\longleftarrow}&E&
   \underset{\Phi_{|\eta|}}{\longrightarrow}&
   \mathbf{P}^{1}&\qquad\mathbf{P}^{1}.&\\
 \end{array}
 $$
\end{remark}
\begin{remark}\label{121832_11Dec16}
 Here we give a description of special fibers of these fibrations of genus $3$.
 For simplicity, we assume that $f:Y\rightarrow{B}$ does not have a singular fiber of type $\mathstrut_{1}I_{4}$.
 As we have seen in Section~\ref{101650_14Nov16}, there are $6$ singular fibers of type $\mathstrut_{1}I_{2}$ in $f:Y\rightarrow{B}$.
 In this case, the number of special fibers of the above bielliptic fibrations of genus $3$ are summarized in the following table, where the $j$-functions is for the corresponding fibration of curves of genus $1$, which is computed in Proposition~\ref{143201_14Oct16} for $f:Y\rightarrow{B}$.\medskip\\
 \begin{tabular}{|l||c|c|c|c|}
  \hline
  Genus $3$ fibration
  &$A'=C^{(2)}/(\sigma^{(2)}\circ\kappa)\rightarrow\mathbf{P}^{1}$
      &$W\rightarrow{B}$
	  &$\mathcal{D}\rightarrow{E}$
	      &$P^{\sim}\rightarrow\mathbf{P}^{1}$\\
  \hline
  \hline
  (1)\ Dual fibers $C^{\vee}$&$1$&$1$&$4$&$4$\\
  \hline
  (2)\ Hyperelliptic fibers&&$3$&$12$&$6$\\
  (2)'\ Multiple fibers&$3$& & & \\
  \hline
  (3)\ Singular fibers
  &$3$&$6$&$24$&$12$\\
  \hline
  \hline
  Genus $1$ fibration
  &$\Km{(A)}\overset{g_{1}}{\rightarrow}\mathbf{P}^{1}$
      &$Y\rightarrow{B}$
	  &$Y\times_{B}{E}\rightarrow{E}$
	      &$\Km{(P)}\overset{g_{2}}{\rightarrow}\mathbf{P}^{1}$\\
  \hline
  Degree of $j$-function&$6$&$12$&$48$&$24$\\
  \hline
 \end{tabular}
 \medskip\\
 The double covering $\Phi_{|\mathcal{O}_{B}(2\eta)|}:B\rightarrow{\mathbf{P}^{1}}$ has $4$ ramification points.
 One of them is the canonical point $\eta$, at which the fiber is the dual bielliptic curve $C^{\vee}=D_{\eta}$.
 The other $3$ ramification points correspond to hyperelliptic fibers $D_{\xi}=W\cap{P_{\xi}}$, because the involution $\kappa$ acts on $D_{\xi}$ as the hyperelliptic involution.
 In this case, the quotient $D_{\xi}/(\sigma^{(2)}\circ\kappa)$ is a nonsingular projective curve of genus $2$, which gives a multiple fiber of $A'=C^{(2)}/(\sigma^{(2)}\circ\kappa)\rightarrow\mathbf{P}^{1}=|\mathcal{O}_{B}(2\eta)|$.
 The ramification points of $\Phi_{|\eta|}:E\rightarrow{\mathbf{P}^{1}}$ corresponds to $4$ dual fibers $C^{\vee}$ in $\mathcal{D}\rightarrow{E}$.
 The number of singular fibers and hyperelliptic fibers for the pencil $P^{\sim}\rightarrow\mathbf{P}^{1}=|\mathcal{O}_{P}(D_{p_{0}})|$ is already computed in \cite[Section~3]{BPS}.
\end{remark}
\begin{corollary}\label{104804_30Nov16}
 If $D_{\xi}=W\cap{P_{\xi}}$ is nonsingular, then
 $\Prym{(D_{\xi}/F_{\xi})}\simeq{A}$,
 where $F_{\xi}$ denotes the fiber of $f:Y\rightarrow{B}$ at $\xi\in{B}$.
\end{corollary}
\begin{proof}
 The bielliptic involution on $D_{\xi}$ is compatible with the involution $(-1)_{P}$ on $P$ by the embedding $D_{\xi}\simeq{D_{p}}\subset{P}$.
 By \cite[Proposition~(1.10)]{B}, the embedding of $D_{\xi}$ to the abelian surface is essentially unique, and it should be the embedding to the dual
$\Prym{(D_{\xi}/F_{\xi})^{\vee}}$ of the Prym variety $\Prym{(D_{\xi}/F_{\xi})}$.
 Hence $P\simeq\Prym{(D_{\xi}/F_{\xi})^{\vee}}$ and
 $A\simeq\Prym{(D_{\xi}/F_{\xi})}$.
\end{proof}
Let $\mathcal{M}$ be the set of isomorphism class of nonsingular bielliptic curve $C\rightarrow{E}$ of genus $3$, where two bielliptic curves $C\rightarrow{E}$ and $C'\rightarrow{E'}$ is isomorphic if there is an isomorphism $C\simeq{C'}$ which commutes with the bielliptic involutions.
We denote the isomorphism class by
$[C\rightarrow{E}]\in\mathcal{M}$.
Let $\mathcal{M}_{3}$ be the moduli space of nonsingular projective curve of genus $3$.
Then the map
$\mathcal{M}\rightarrow\mathcal{M}_{3};
[C\rightarrow{E}]\mapsto[C]$
is birational to the image, and $\mathcal{M}$ is a rational variety of dimension $4$ by \cite{BC}.
Let $\mathcal{A}_{1}\simeq\mathbf{A}^{1}$ be the moduli space of elliptic curves.
Then the map
$\mathcal{M}\rightarrow\mathcal{A}_{1};
[C\rightarrow{E}]\mapsto[J(E)]$
is computed by the $j$-function
$j:\mathcal{M}\rightarrow\mathbf{P}^{1};
[C\rightarrow{E}]\mapsto{j(E)}$.
Let $\mathcal{A}_{2}^{(1,2)}$ be the moduli space of $(1,2)$-polarized abelian surface, which is a rational variety of dimension $3$ by \cite{BL2}.
The Prym map
$$
\mathcal{M}\longrightarrow\mathcal{A}_{2}^{(1,2)};\
[C\rightarrow{E}]\longmapsto\Prym{(C/E)}
$$
is dominant, and it has $1$-dimensional fiber.
Let $\mathcal{M}(A)$ be the fiber of the Prym map at $(A,[\mathcal{L}_{A}])\in\mathcal{A}_{2}^{(1,2)}$, and let $(P,[\mathcal{L}_{P}])$ be the dual of $(A,[\mathcal{L}_{A}])$, where $\mathcal{L}_{A}$ and $\mathcal{L}_{P}$ denote the invertible sheaves giving their polarizations.
\begin{lemma}\label{110908_14Dec16}
 If $\Aut{(P,0,[\mathcal{L}_{P}])}=\{\pm1_{P}\}$, then the members of the linear system $|\mathcal{L}_{P}|$ gives a covering $|\mathcal{L}_{P}|\dashrightarrow\mathcal{M}(A)$ of degree $4$.
\end{lemma}
\begin{proof}
 Let $D\rightarrow{F}$ be a bielliptic curve in $\mathcal{M}(A)$.
 By \cite[Proposition~(1.10)]{B}, the curve $D$ is embedded in
 $\Prym{(D/F)}^{\vee}\simeq{P}$ as a member of the linear system $|\mathcal{L}_{P}|$, and if $\Aut{(P,0,[\mathcal{L}_{P}])}=\{\pm1_{P}\}$, then the embedding is uniquely determined up to the translation by a base point of $|\mathcal{L}_{P}|$.
 Since $|\mathcal{L}_{P}|$ has $4$ base points, a general bielliptic curve
 $[D\rightarrow{F}]\in\mathcal{M}(A)$ gives $4$ different members in $|\mathcal{L}_{P}|$.
 Hence the covering $|\mathcal{L}_{P}|\dashrightarrow\mathcal{M}(A)$ is of degree $4$.
\end{proof}
\begin{proposition}\label{175005_11Dec16}
 The period map
 $$
 \mathcal{M}\longrightarrow\mathcal{A}_{1}\times\mathcal{A}_{2}^{(1,2)};\
 [C\rightarrow{E}]\longmapsto(J(E),\Prym{(C/E)})
 $$
 is quasi-finite of degree $6$.
\end{proposition}
\begin{proof}
 Let $(A,[\mathcal{L}_{A}])\in\mathcal{A}_{2}^{(1,2)}$ be a $(1,2)$-polarized abelian variety which is in the image of the Prym map
 $\mathcal{M}\rightarrow\mathcal{A}_{2}^{(1,2)}$,
 and let $(P,[\mathcal{L}_{P}])$ be the dual of $(A,[\mathcal{L}_{A}])$.
 As we noted in Remark~\ref{121832_11Dec16}, by Proposition~\ref{143201_14Oct16}, the $j$-function
 $$
 j:|\mathcal{L}_{P}|\longrightarrow\mathbf{P}^{1};\
 D\subset{P}\longmapsto{j(D/(-1)_{P})}
 $$
 of the fibration $g_{2}:\Km{(P)}\rightarrow\mathbf{P}^{1}=|\mathcal{L}_{P}|$
 is of degree $24$.
 Since the map $|\mathcal{L}_{P}|\dashrightarrow\mathcal{M}(A)$ is of degree $4$ by Lemma~\ref{110908_14Dec16}, the map
 $$
 j:\mathcal{M}(A)\longrightarrow\mathbf{P}^{1};\
 [D\rightarrow{F}]\longmapsto{j(F)}
 $$
 is of degree $6$.
\end{proof}
\begin{remark}
 The Proposition~\ref{175005_11Dec16} can be proved by another way.
 Let $E\in\mathcal{A}_{1}$ be a general elliptic curve, and let $(A,[\mathcal{L}_{A}])\in\mathcal{A}_{2}^{(1,2)}$ be a general $(1,2)$-polarized abelian surface.
 Let $K(A)\subset{A}$ be the subgroup
 $$
 K(A)=\Ker{(A\longrightarrow{\Pic^{(0)}(A)};\
 x\longmapsto[t_{x}^{*}\mathcal{L}_{A}\otimes\mathcal{L}_{A}^{\vee}])}.
 $$
 For an isomorphism
 $
 \chi:
 J(E)_{2}\simeq{K(A)}
 $
 of finite groups,
 we set a finite subgroup by
 $$
 K_{\chi}=\{(x,y)\in{J(E)\times{A}}\mid{x\in{J(E)_{2},\
 y\in{K(A)},\ y=\chi(x)}}\}.
 $$
 Then
 $
 J_{\chi}=(J(E)\times{A})/K_{\chi}
 $
 is a principally polarized abelian variety of dimension $3$.
 When $J_{\chi}$ is the Jacobian variety of the nonsingular projective curve $C_{\chi}$, we have a bielliptic curve $C_{\chi}\rightarrow{E}$ and an isomorphism $A\simeq\Prym{(C_{\chi}/E)}$.
 Since we have $6$ choices of the isomorphism $\chi$, the period map is of degree $6$.
\end{remark}
\subsection{Dual families}
We give an explicit equation of the family $\{D_{\xi}\rightarrow{F_{\xi}}\}_{\xi}$ of bielliptic curves for a bielliptic curve $C\rightarrow{E}$.
We assume that $C$ is not hyperelliptic.
Then we can denote the equation of the canonical model of $C$ by
$$
C=\{[x:y:z]\in\mathbf{P}^{2}\mid
(z^{2}+S(x,y))^{2}=T(x,y)\},
$$
and the bielliptic involution $\sigma$ is given by $z\mapsto{-z}$.
The quotient $E=C/\sigma$ is given by
$$
\pi:C\longrightarrow
E=\{[x:y:w]\in\mathbf{P}(1,1,2)\mid
(w+S(x,y))^{2}=T(x,y)\};\
[x:y:z]\longmapsto[x:y:w]=[x:y:z^{2}].
$$
\begin{lemma}\label{095421_30Nov16}
 The dual bielliptic curve $C^{\vee}$ is a non-hyperelliptic curve defined by
 $$
 C^{\vee}\simeq
 \{[x:y:z]\in\mathbf{P}^{2}\mid
 (z^{2}+S(x,y))^{2}=S(x,y)^{2}-T(x,y).
 \}
 $$
\end{lemma}
\begin{proof}
 Let
 $
 \Phi:
 E\rightarrow\mathbf{P}^{1};\
 [x:y:w]\mapsto[x:y]
 $
 be the projection.
 Then the canonical point is given by
 $\eta=[\Phi^{*}\mathcal{O}_{\mathbf{P}^{1}}(1)]\in{B=\Pic^{(2)}{(E)}}$,
 because $\Omega_{C}^{1}\simeq(\Phi\circ\pi)^{*}\mathcal{O}_{\mathbf{P}^{1}}(1)$.
 If $q+q'\in{C^{\vee}}\subset{C^{(2)}}$,
 then $\pi(q)+\pi(q')\in|\Phi^{*}\mathcal{O}_{\mathbf{P}^{1}}(1)|$,
 hence $\Phi\circ\pi(q)=\Phi\circ\pi(q')$.
 When we denote by $q=[x:y:z]$ and $q'=[x:y:z']$,
 we have
 $$
 \begin{cases}
  z^{2}+z'^{2}=-2S(x,y),\\
  z^{2}z'^{2}=S(x,y)^{2}-T(x,y),
 \end{cases}
 $$
 because $\pi(q)+\pi(q')=\Phi^{-1}([x:y])$.
 Then the isomorphism is given by
 $$
 \begin{array}{ccl}
  C^{\vee}&\longrightarrow&
  \{[x:y:z]\in\mathbf{P}^{2}\mid
  (z^{2}+S(x,y))^{2}=S(x,y)^{2}-T(x,y)\};\\
  \mathstrut[x:y:z]
   +[x:y:z']&\longmapsto&
  [x:y:\frac{z+z'}{\sqrt{2}}].
 \end{array}
 $$
\end{proof}
By a suitable change of the coordinate, we assume that
$$
\begin{cases}
 S(x,y)=s_{0}x^{2}+s_{1}xy+s_{2}y^{2},\\
 T(x,y)=x^{3}y+t_{1}x^{2}y^{2}+t_{2}xy^{3}+t_{3}y^{4}.
\end{cases}
$$
Then $E$ is isomorphic to the plane cubic curve by
$$
E\simeq
\{[x:y:z]\in\mathbf{P}^{2}\mid
y^{2}z=x^{3}+t_{1}x^{2}z+t_{2}xz^{2}+t_{3}z^{3}\};\
[x:y:w]\longmapsto
[x:y:z]=[xy:w+S(x,y):y^{2}].
$$
Let $p_{0}=[0:1:0]\in{E\subset\mathbf{P}^{2}}$ be the point on the cubic curve.
We remark that $\eta=[\mathcal{O}_{E}(2p_{0})]\in\Pic^{(2)}{(E)}$ is the canonical point of the covering $\pi:C\rightarrow{E}$.
Then for $p=[a:b:1]\in{E}$, we compute the equation of the canonical model of
$$
D_{\xi}=W\cap{P_{\xi}}\simeq\{q+q'\in{C^{(2)}}\mid\pi(q)+\pi(q')\in|\xi|\},
$$
where $\xi=[\mathcal{O}_{E}(p+p_{0})]\in\Pic^{(2)}{(E)}$.
\begin{lemma}\label{172705_10Dec16}
 $$
 D_{\xi}\simeq
 \{[x:y:z]\in\mathbf{P}^{2}\mid
 c_{0}z^{4}-2c_{1}(x,y)z^{2}+c_{2}(x,y)=0\}
 $$
 for general $\xi\in\Pic^{(2)}{(E)}$, where
 $$
 \begin{cases}
  c_{0}=-a^{3}+({s}_{1}^{2}-4{s}_{0}{s}_{2}-{t}_{1})a^{2}
  -(4{s}_{0}{s}_{2}{t}_{1}-2{s}_{0}{s}_{1}{t}_{2}+4{s}_{0}^{2}{t}_{3}-2{s}_{1}{s}_{2}+{t}_{2})a\\
  \hspace*{230pt}
  +{s}_{0}^{2}{t}_{2}^{2}-4{s}_{0}^{2}{t}_{1}{t}_{3}-2{s}_{0}{s}_{2}{t}_{2}+4{s}_{0}{s}_{1}{t}_{3}+{s}_{2}^{2}-{t}_{3},\\
  c_{1}(x,y)=(2{s}_{0}a^{2}+(2{s}_{0}{t}_{1}-{s}_{1})a+{s}_{0}{t}_{2}-{s}_{2})x^{2}+2({s}_{1}a^{2}+({s}_{0}{t}_{2}+{s}_{2})a+2{s}_{0}{t}_{3})xy\\
  \hspace*{130pt}
  +(2{s}_{2}a^{2}+(2{s}_{2}{t}_{1}-{s}_{1}{t}_{2}+4{s}_{0}{t}_{3})a-{s}_{0}{t}_{2}^{2}+4{s}_{0}{t}_{1}{t}_{3}+{s}_{2}{t}_{2}-2{s}_{1}{t}_{3})y^{2},\\
  c_{2}(x,y)=x^{4}-4ax^{3}y-2(2{t}_{1}a+{t}_{2})x^{2}y^{2}-4({t}_{2}a+2{t}_{3})xy^{3}-(4{t}_{3}a-{t}_{2}^{2}+4{t}_{1}{t}_{3})y^{4}.\\
 \end{cases}
 $$
\end{lemma}
\begin{proof}
 Since $\Omega_{C^{(2)}}^{2}\vert_{D_{\xi}}\simeq\Omega_{D_{\xi}}^{1}$, we compute the image of $D_{\xi}\subset{C^{(2)}}$ by the canonical morphism
 $$
 \Phi_{|\Omega_{C^{(2)}}^{2}|}:\
 C^{(2)}\longrightarrow\mathbf{P}^{2};\
 [x:y:z]+[x':y':z']\longmapsto
 [x'z-xz':y'z-yz':xy'-x'y].
 $$
 For $q+q'=[x:1:z]+[x':1:z']\in{D_{\xi}\subset{C^{(2)}}}$, we set
 $$
 m=\frac{z^{2}+S(x,1)-z'^2-S(x',1)}{x-x'},\quad
 u=\frac{x'z-xz'}{z-z'},\quad
 v=\frac{x-x'}{z-z'},
 $$
 where $u$ and $v$ give the coordinate of the point
 $\Phi_{|\Omega_{C^{(2)}}^{2}|}(q+q')=[u:1:v]$.
 Then these variables have relations
 $$
 \begin{cases}
  (m-s_{0}p_{1}-s_{1})v=z+z',\\
  (z+z')v=(x+x')-2u.
 \end{cases}
 $$
 Since $\pi(q)+\pi(q')\in|\xi|=|\mathcal{O}_{E}(p+p_{0})|$, we have a linear equivalence
 $$
 [x:z^{2}+S(x,1):1]+[x':z'^{2}+S(x',1):1]\sim
 [a:b:0]+[0:1:0]
 $$
 on the plane cubic curve $E\subset\mathbf{P}^{2}$,
 hence we have relations
 $$
 \begin{cases}
  z^{2}=m(x-a)-b-S(x,1),\\
  z'^{2}=m(x'-a)-b-S(x',1),\\
  x+x'=m^{2}-t_{1}-a,\\
  xx'=am^{2}+2bm+t_{2}+a(t_{1}+a).
 \end{cases}
 $$
 By eliminating the variable $x,x',z,z',m$ from these relations, we have a relation of $u$ and $v$.
 When we eliminate $b^{2}$ from this relation by $b^{2}=a^{3}+t_{1}a^{2}+t_{2}a+t_{3}$, we have the relation
 $$
 c_{0}v^{4}-2c_{1}(u,1)v^{2}+c_{2}(u,1)=0.
 $$
\end{proof}
\begin{corollary}\label{191514_10Dec16}
 The $j$-invariant of the canonical divisor $K(D_{\xi}/F_{\xi})$
 of $Y(D_{\xi}/F_{\xi})$ is
 $$
 j(K(D_{\xi}/F_{\xi}))=2^{8}\cdot
 \frac{(({t}_{1}^{2}-3{t}_{2})a^{2}+({t}_{1}{t}_{2}-9{t}_{3})a+{t}_{2}^{2}-3{t}_{1}{t}_{3})^{3}}
 {\disc{(\tau)}\cdot{\tau(a)^{2}}},
 $$
 where
 $$
 \tau(x)=T(x,1)={x^{3}+x^{2}{t}_{1}+x{t}_{2}+{t}_{3}}
 $$
 and
 $$
 \disc{(\tau)}={t}_{1}^{2}{t}_{2}^{2}-4{t}_{1}^{3}{t}_{3}-4{t}_{2}^{3}+18{t}_{1}{t}_{2}{t}_{3}-27{t}_{3}^{2}
 $$
 denotes the discriminant of the equation $\tau(x)=0$.
\end{corollary}
\begin{proof}
 By Lemma~\ref{095421_30Nov16}, the dual bielliptic curve
 $D_{\xi}^{\vee}\rightarrow{K(D_{\xi}/F_{\xi})}$ of $D_{\xi}\rightarrow{F_{\xi}}$ is defined by
 $$
 (z^{2}-\frac{c_{1}(x,y)}{c_{0}})^{2}
 =\frac{c_{2}(x,y)}{c_{0}}.
 $$
 We can directly compute the $j$-invariant of $K(D_{\xi}/F_{\xi})$ from the polynomial $c_{2}(x,y)$, because $K(D_{\xi}/F_{\xi})$ is the double covering of $\mathbf{P}^{1}$ branched along
 $\{[x:y]\in\mathbf{P}^{1}\mid{c_{2}(x,y)=0}\}$.
\end{proof}
\begin{remark}\label{153107_13Jan17}
 By Corollary~\ref{191514_10Dec16}, the $j$-function
 $$
 B=\Pic^{(2)}{(E)}\longrightarrow\mathbf{P}^{1};\
 \xi\longmapsto
 j(K(D_{\xi}/F_{\xi}))
 $$
 factors through
 $\Phi_{|\mathcal{O}_{B}(2\eta)|}:B\rightarrow
 \mathbf{P}^{1}=\mathbf{P}^{1}(a\text{-line})$.
 Then the covering
 $$
 j_{K}:\
 \mathbf{P}^{1}(a\text{-line})\longrightarrow\mathbf{P}^{1};\
 a\longmapsto
 2^{8}\cdot
 \frac{(({t}_{1}^{2}-3{t}_{2})a^{2}+({t}_{1}{t}_{2}-9{t}_{3})a+{t}_{2}^{2}-3{t}_{1}{t}_{3})^{3}}
 {\disc{(\tau)}\cdot{\tau(a)^{2}}},
 $$
 has $3$ branched points $j=0,12^{3},\infty$.
 The roots of
 $$
 ({t}_{1}^{2}-3{t}_{2})a^{2}+({t}_{1}{t}_{2}-9{t}_{3})a+{t}_{2}^{2}-3{t}_{1}{t}_{3}=0
 $$
 gives points of the ramification index $3$,
 ans the roots of
 $$
 ((2{t}_{1}^{3}-9{t}_{1}{t}_{2}+27{t}_{3})a^{3}+3({t}_{1}^{2}{t}_{2}-6{t}_{2}^{2}+9{t}_{1}{t}_{3})a^{2}-3({t}_{1}{t}_{2}^{2}-6{t}_{1}^{2}{t}_{3}+9{t}_{2}{t}_{3})a-(2{t}_{2}^{3}-9{t}_{1}{t}_{2}{t}_{3}+27{t}_{3}^{2}))
 \tau(a)=0
 $$
 gives points of the ramification index $2$.
 Hence the covering $j_{K}$ has $10$ ramification points on the $a$-line.
\end{remark}
\begin{corollary}\label{170044_12Jan17}
 The $j$-invariant of $F_{\xi}$ is
 $$
 j(F_{\xi})=2^{8}\cdot
 \frac{(({\check{t}}_{1}^{2}-3{\check{t}}_{2})a^{2}+({\check{t}}_{1}{\check{t}}_{2}-9{\check{t}}_{3})a+{\check{t}}_{2}^{2}-3{\check{t}}_{1}{\check{t}}_{3})^{3}}
 {\disc{(\check{\tau})}\cdot{\check{\tau}(a)^{2}}},
 $$
 where
 $$
 \begin{cases}
  \check{t}_{1}=-{s}_{1}^{2}+4{s}_{0}{s}_{2}+{t}_{1},\\
  \check{t}_{2}=4{s}_{0}{s}_{2}{t}_{1}-2{s}_{0}{s}_{1}{t}_{2}+4{s}_{0}^{2}{t}_{3}-2{s}_{1}{s}_{2}+{t}_{2},\\
  \check{t}_{3}=-{s}_{0}^{2}{t}_{2}^{2}+4{s}_{0}^{2}{t}_{1}{t}_{3}+2{s}_{0}{s}_{2}{t}_{2}-4{s}_{0}{s}_{1}{t}_{3}-{s}_{2}^{2}+{t}_{3}
 \end{cases}
 $$
 and
 $$
 \check{\tau}(x)=x^{3}+{\check{t}}_{1}x^{2}+{\check{t}}_{2}x+{\check{t}}_{3}.
 $$
\end{corollary}
\begin{proof}
 By Lemma~\ref{172705_10Dec16}, the curve $D_{\xi}$ is defined by
 $$
 (z^{2}-\frac{c_{1}(x,y)}{c_{0}})^{2}
 =\frac{-c_{0}c_{2}(x,y)+c_{1}(x,y)^{2}}{c_{0}^{2}},
 $$
 hence
 $F_{\xi}$ is the double covering of $\mathbf{P}^{1}$ branched along
 $\{[x:y]\in\mathbf{P}^{1}\mid{c_{1}(x,y)^{2}-c_{0}c_{2}(x,y)=0}\}$, and we can directly compute the $j$-invariant $j(F_{\xi})$.
\end{proof}
\begin{remark}\label{153122_13Jan17}
 The $j$-function in Corollary~\ref{170044_12Jan17} is already computed at Proposition~\ref{143201_14Oct16}, and
 it also factors through the $\mathbf{P}^{1}(a\text{-line})$;
 $$
 j_{F}:\
 \mathbf{P}^{1}(a\text{-line})\longrightarrow\mathbf{P}^{1};\
 a\longmapsto
 2^{8}\cdot
 \frac{(({\check{t}}_{1}^{2}-3{\check{t}}_{2})a^{2}+({\check{t}}_{1}{\check{t}}_{2}-9{\check{t}}_{3})a+{\check{t}}_{2}^{2}-3{\check{t}}_{1}{\check{t}}_{3})^{3}}
 {\disc{(\check{\tau})}\cdot{\check{\tau}(a)^{2}}}.
 $$
\end{remark}
\begin{remark}
 The following conditions are equivalent;
 \begin{enumerate}
  \item $C$ is nonsingular,
  \item $C^{\vee}$ is nonsingular,
  \item $\disc{(\tau)}\cdot\disc{(\check{\tau})}\neq0$.
 \end{enumerate}
\end{remark}

\section{Hodge structure}\label{221934_29Nov16}
In this section, We describe the relation between the Hodge structures of the elliptic surface $Y(C/E)$ and the abelian surface $A=J(C)/\pi^{*}J(E)$.\par
Let $S\rightarrow{C^{(2)}}$ be the blow-up at the $6$ points $q_{i}+q_{j}\in{C^{(2)}}$, where $q_{1},\dots,q_{4}$ denote the ramification points of $\pi:C\rightarrow{E}$.
Then we have a morphism $\rho:S\rightarrow{Y}$ in the commutative diagram
$$
\begin{array}{ccc}
 C^{(2)}&\longleftarrow&{S}\\
 \downarrow& &\quad\downarrow\rho\\
 C^{(2)}/\sigma^{(2)}&\underset{\nu}{\longleftarrow}&Y.\\
\end{array}
$$
We denote by $E_{ij}\subset{S}$ the exceptional curve over the point $q_{i}+q_{j}$.
Since the curve $\Delta_{\sigma}$ in the proof of Lemma~\ref{115601_2Dec16} does not meet the blow-up center of $S\rightarrow{C^{(2)}}$, we consider $\Delta_{\sigma}$ as a curve in $S$.
Then $\Delta_{\sigma}+\sum_{1\leq{i}<{j}\leq{4}}E_{ij}$ is the ramification divisor of the finite double covering $\rho$.
We denote by $\tilde{\Gamma}_{i}$ the proper transform of
$\{q_{i}+q\in{C^{(2)}}\mid{q\in{C}}\}$
in $S$.
Then $\rho(\tilde{\Gamma}_{i})\subset{Y}$ is the component of the ramification divisor of the double covering $\phi\circ\nu:Y\rightarrow{E^{(2)}}$.
We set a finite subset $\Pi_{Y}\subset{H^{2}(Y,\mathbf{Z})}$ by
$$
\Pi_{Y}=\{[K_{Y}],
[\rho(\Delta_{\sigma})],[\rho(\tilde{\Gamma}_{1})],\dots,[\rho(\tilde{\Gamma}_{4})],
[\rho(E_{12})],[\rho(E_{13})],[\rho(E_{14})],
[\rho(E_{34})],[\rho(E_{24})],[\rho(E_{34})]
\}.
$$
\begin{theorem}\label{104530_30Nov16}
 The Hodge structure $(H^{2}(A,\mathbf{Z}),\langle\ ,\ \rangle_{A},[\iota(C)])$ with the symmetric form and the ample class is determined by $(H^{2}(Y,\mathbf{Z}),\langle\ ,\ \rangle_{Y},\Pi_{Y})$.
 Conversely, the Hodge structure $(H^{2}(Y,\mathbf{Z}),\langle\ ,\ \rangle_{Y},\Pi_{Y})$ with the symmetric form and the finite set of classes is determined by $(H^{2}(A,\mathbf{Z}),\langle\ ,\ \rangle_{A},[\iota(C)])$.
\end{theorem}
The idea of the proof of this theorem is to compare Hodge structures $H^{2}(Y,\mathbf{Z})$ and $H^{2}(A,\mathbf{Z})$ in $H^{2}(S,\mathbf{Z})$, because we have the double covering $\lambda:S\rightarrow{A}$ by Lemma~\ref{115601_2Dec16}.
In fact, we will see that $\lambda^{*}H^{2}(A,\mathbf{Z})$ is contained in the primitive closure of $\rho^{*}H^{2}(Y,\mathbf{Z})$ in $H^{2}(S,\mathbf{Z})$.
In subsection~\ref{154110_6Dec16}, we prepare the explicit description for the integral basis of $H^{2}(Y,\mathbf{Z})$.
In subsection~\ref{154259_6Dec16}, we construct the Hodge structure $(H^{2}(A,\mathbf{Z}),\langle\ ,\ \rangle_{A},[\iota(C)])$ from the Hodge structure $(H^{2}(Y,\mathbf{Z}),\langle\ ,\ \rangle_{Y},\Pi_{Y})$.
In subsection~\ref{154538_6Dec16}, we construct the Hodge structure $(H^{2}(Y,\mathbf{Z}),\langle\ ,\ \rangle_{Y},\Pi_{Y})$ from the Hodge structure $(H^{2}(A,\mathbf{Z}),\langle\ ,\ \rangle_{A},[\iota(C)])$.
Theorem~\ref{104530_30Nov16} is proved by Lemma~\ref{161526_8Dec16} and Lemma~\ref{132213_8Dec16}.
\subsection{Lattice $H^{2}(Y,\mathbf{Z})$}\label{154110_6Dec16}
Let $\lambda_{1}:H^{1}(C,\mathbf{Z}){\longrightarrow}
H^{1}(C^{(2)},\mathbf{Z})$
and
$\lambda_{2}:\bigwedge^{2}H^{1}(C,\mathbf{Z}){\longrightarrow}
H^{2}(C^{(2)},\mathbf{Z})$
be homomorphisms defined by
$$
\epsilon^{*}\lambda_{1}(\gamma)
=\mathrm{pr}_{1}^{*}\gamma+\mathrm{pr}_{2}^{*}\gamma,\quad
\epsilon^{*}\lambda_{2}(\gamma\wedge\gamma')
=\mathrm{pr}_{1}^{*}\gamma\cup\mathrm{pr}_{2}^{*}\gamma'
+\mathrm{pr}_{2}^{*}\gamma\cup\mathrm{pr}_{1}^{*}\gamma'
$$
for $\gamma,\gamma'\in{H^{1}(C,\mathbf{Z})}$, where
$\epsilon:C\times{C}\rightarrow{C^{(2)}}$ denotes the natural double covering, and
$\mathrm{pr}_{i}:C\times{C}\rightarrow{C}$ denotes the $i$-th projection.
Then $\lambda_{1}$ is an isomorphism, and $\lambda_{2}$ is injective.
We denote by $\widehat{\delta}\in{H^{2}(C^{(2)},\mathbf{Z})}$ the class of the divisor
$
\{q+q'\in{C^{(2)}}\mid{q'\in{C}}\},
$
which does not depend on $q\in{C}$.
Then $H^{2}(C^{(2)},\mathbf{Z})$ is generated by $\widehat{\delta}$ and
$\lambda_{2}(\bigwedge^{2}H^{1}(C,\mathbf{Z}))$ by \cite{Ma}.
We remark that
$$
\lambda_{1}(\gamma)\cup\lambda_{1}(\gamma')
=\lambda_{2}(\gamma\wedge\gamma')+\langle\gamma,\gamma'\rangle_{C}\widehat{\delta}
\in{H^{2}(C^{(2)},\mathbf{Z})}
$$
for $\gamma,\gamma'\in{H^{1}(C,\mathbf{Z})}$, where
$\langle\ ,\ \rangle_{C}$ denotes the alternating form on $H^{1}(C,\mathbf{Z})$.
\begin{lemma}\label{123412_3Nov16}
 $$
 \langle\lambda_{2}(\gamma_{1}\wedge\gamma'_{1}),
 \lambda_{2}(\gamma_{2}\wedge\gamma'_{2})\rangle_{C^{(2)}}
 =-\det{
 \left(
 \begin{array}{cc}
  \langle\gamma_{1},\gamma_{2}\rangle_{C}&\langle\gamma_{1},\gamma'_{2}\rangle_{C}\\
  \langle\gamma'_{1},\gamma_{2}\rangle_{C}&\langle\gamma'_{1},\gamma'_{2}\rangle_{C}\\
 \end{array}
 \right)}
 $$
 for
 $\gamma_{1},\gamma'_{1},\gamma_{2},\gamma'_{2}
 \in{H^{1}(C,\mathbf{Z})}$.
\end{lemma}
\begin{proof}
 This follows from
 $\langle\epsilon^{*}x,\epsilon^{*}y\rangle_{C\times{C}}
 =2\langle{x},y\rangle_{C^{(2)}}$
 for $x,y\in{H^{2}(C^{(2)},\mathbf{Z})}$.
\end{proof}
We fix a symplectic basis
$
\alpha_{1},\dots,\alpha_{3},\beta_{1},\dots,\beta_{3}\in
H^{1}(C,\mathbf{Z})
$
which satisfies
$$
\sigma^{*}\alpha_{1}=-\alpha_{1},\
\sigma^{*}\alpha_{2}=\alpha_{3},\
\sigma^{*}\beta_{1}=-\beta_{1},\
\sigma^{*}\beta_{2}=\beta_{3}.
$$
We set elements in $H^{2}(C^{(2)},\mathbf{Z})$ by
$$
\begin{cases}
 \delta_{0}=\lambda_{2}(\alpha_{1},\beta_{1}),\\
 \delta_{1}=\lambda_{2}(\alpha_{2},\beta_{2})+\lambda_{2}(\alpha_{3},\beta_{3}),\\
 \delta_{2}=\lambda_{2}(\alpha_{2},\beta_{3})+\lambda_{2}(\alpha_{3},\beta_{2}),
\end{cases}
\qquad
\begin{cases}
 \delta_{3}=\lambda_{2}(\alpha_{1},\alpha_{2}-\alpha_{3}),\\
 \delta_{4}=\lambda_{2}(\beta_{1},\beta_{2}-\beta_{3}),\\
 \delta_{5}=\lambda_{2}(\alpha_{1},\beta_{2}-\beta_{3}),\\
 \delta_{6}=\lambda_{2}(\alpha_{2}-\alpha_{3},\beta_{1}).
\end{cases}
$$
\begin{lemma}\label{123002_3Nov16}
 $\widehat{\delta},\delta_{0},\dots,\delta_{6}$
 form a $\mathbf{Z}$-basis of the invariant part $H^{2}(C^{(2)},\mathbf{Z})^{\sigma}$ of the $\sigma^{*}$-action, and the intersection matrix is
 $$
 \langle
 \left(
 \begin{array}{c}
  \widehat{\delta}\\
  \delta_{0}\\
  \vdots\\
  \delta_{6}\\
 \end{array}
 \right),
 \left(
 \begin{array}{cccc}
  \widehat{\delta}&\delta_{0}&\cdots&\delta_{6}\\
 \end{array}
 \right)
 \rangle_{C^{(2)}}={\scriptscriptstyle
 \left(
 \begin{array}{cccccccc}
  1&0&0&0&0&0&0&0\\
  0&-1&0&0&0&0&0&0\\
  0&0&-2&0&0&0&0&0\\
  0&0&0&-2&0&0&0&0\\
  0&0&0&0&0&-2&0&0\\
  0&0&0&0&-2&0&0&0\\
  0&0&0&0&0&0&0&-2\\
  0&0&0&0&0&0&-2&0\\
 \end{array}
 \right)}.
 $$
\end{lemma}
\begin{proof}
 This follows from a computation by Lemma~\ref{123412_3Nov16}.
\end{proof}
\begin{lemma}\label{120929_3Nov16}
 The classes of divisors $\Delta_{\sigma}$ and $K_{C^{(2)}}$ on $C^{(2)}$ are
 $[\Delta_{\sigma}]=\widehat{\delta}+\delta_{0}-\delta_{2}$ and
 $[K_{C^{(2)}}]=3\widehat{\delta}+\delta_{0}+\delta_{1}$
 in
 $H^{2}(C^{(2)},\mathbf{Z})$.
\end{lemma}
\begin{proof}
 Since the class of the graph
 $\{(p,p')\in{C\times{C}}\mid{p'=\sigma(p)}\}$
 is
 $$
 \epsilon^{*}\widehat{\delta}
 +\pr_{1}^{*}\alpha_{1}\cup\pr_{2}^{*}\beta_{1}
 +\pr_{2}^{*}\alpha_{1}\cup\pr_{1}^{*}\beta_{1}
 -\pr_{1}^{*}\alpha_{2}\cup\pr_{2}^{*}\beta_{3}
 -\pr_{2}^{*}\alpha_{2}\cup\pr_{1}^{*}\beta_{3}
 -\pr_{1}^{*}\alpha_{3}\cup\pr_{2}^{*}\beta_{2}
 -\pr_{2}^{*}\alpha_{3}\cup\pr_{1}^{*}\beta_{2},
 $$
 we have
 $\epsilon^{*}[\Delta_{\sigma}]=\epsilon^{*}(\widehat{\delta}+\delta_{0}-\delta_{2})$.
 Since the class of the diagonal $\Delta_{C}\subset{C\times{C}}$ is
 $$
 \epsilon^{*}\widehat{\delta}
 -\pr_{1}^{*}\alpha_{1}\cup\pr_{2}^{*}\beta_{1}
 -\pr_{2}^{*}\alpha_{1}\cup\pr_{1}^{*}\beta_{1}
 -\pr_{1}^{*}\alpha_{2}\cup\pr_{2}^{*}\beta_{2}
 -\pr_{2}^{*}\alpha_{2}\cup\pr_{1}^{*}\beta_{2}
 -\pr_{1}^{*}\alpha_{3}\cup\pr_{2}^{*}\beta_{3}
 -\pr_{2}^{*}\alpha_{3}\cup\pr_{1}^{*}\beta_{3},
 $$
 we have
 $$
 \epsilon^{*}[K_{C^{(2)}}]=[K_{C\times{C}}]-[\Delta_{C}]
 =4\epsilon^{*}\widehat{\delta}
 -\epsilon^{*}(\widehat{\delta}-\delta_{0}-\delta_{1})
 =\epsilon^{*}(3\widehat{\delta}+\delta_{0}+\delta_{1}).
 $$
\end{proof}
\begin{corollary}\label{145134_5Nov16}
 $\rho^{*}[K_{Y}]=2\widehat{\delta}+\delta_{1}+\delta_{2}
 \in{H^{2}(S,\mathbf{Z})=H^{2}(C^{(2)},\mathbf{Z})
 \oplus\bigoplus_{1\leq{i}<{j}\leq{4}}
 \mathbf{Z}[E_{ij}]}$.
\end{corollary}
\begin{proof}
 By Lemma~\ref{120929_3Nov16}, we have
 $$
 [K_{S}]=
 3\widehat{\delta}+\delta_{0}+\delta_{1}+
 \sum_{1\leq{i}<{j}\leq{4}}[E_{ij}]
 \in{H^{2}(S,\mathbf{Z})=H^{2}(C^{(2)},\mathbf{Z})
 \oplus\bigoplus_{1\leq{i}<{j}\leq{4}}
 \mathbf{Z}[E_{ij}]}
 $$
 and
 $$
 \rho^{*}[K_{Y}]=[K_{S}]-[\Delta_{\sigma}]-\sum_{1\leq{i}<{j}\leq{4}}[E_{ij}]
 =2\widehat{\delta}+\delta_{1}+\delta_{2}.
 $$
\end{proof}
\begin{lemma}\label{175822_2Nov16}
 The classes
 $$
 \begin{cases}
  \delta_{7}=\widehat{\delta}-[E_{12}]-[E_{13}]-[E_{14}],\\
  \delta_{8}=\widehat{\delta}-[E_{12}]-[E_{23}]-[E_{24}],\\
  \delta_{9}=\widehat{\delta}-[E_{13}]-[E_{23}]-[E_{34}],\\
  \delta_{10}=\widehat{\delta}-[E_{14}]-[E_{24}]-[E_{34}],\\
  \delta_{11}
  =\widehat{\delta}+\delta_{0}-\delta_{2}
  +\sum_{1\leq{i}<{j}\leq{4}}[E_{ij}]
 \end{cases}
 $$
 in
 $H^{2}(S,\mathbf{Z})=H^{2}(C^{(2)},\mathbf{Z})
 \oplus\bigoplus_{1\leq{i}<{j}\leq{4}}
 \mathbf{Z}[E_{ij}]$
 are contained in $\rho^{*}H^{2}(Y,\mathbf{Z})$.
\end{lemma}
\begin{proof}
 Since $\rho\vert_{\tilde{\Gamma}_{i}}:\tilde{\Gamma}_{i}\rightarrow{\rho(\tilde{\Gamma}_{i})}$ is a double covering,
 $\delta_{i+6}=[\tilde{\Gamma}_{i}]=\rho^{*}[\rho(\tilde{\Gamma}_{i})]$
 for $i=1,\dots,4$.
 Since
 $
 \rho(\Delta_{\sigma})+\sum_{1\leq{i}<{j}\leq{4}}\rho(E_{ij})
 $
 is the branch divisor of $\rho$, there is an invertible sheaf $\mathcal{F}$ such that
 $\mathcal{F}^{\otimes2}\simeq\mathcal{O}_{Y}(\rho(\Delta_{\sigma})+\sum_{1\leq{i}<{j}\leq{4}}\rho(E_{ij}))$.
 By Lemma~\ref{120929_3Nov16},
 $$
 2\rho^{*}c_{1}(\mathcal{F})=
 \rho^{*}[\rho(\Delta_{\sigma})+\sum_{1\leq{i}<{j}\leq{4}}\rho(E_{ij})]
 =2([\Delta_{\sigma}]+\sum_{1\leq{i}<{j}\leq{4}}[E_{ij}])
 =2\delta_{11},
 $$
 where $c_{1}(\mathcal{F})\in{H^{2}(Y,\mathbf{Z})}$ denotes the first Chern class of $\mathcal{F}$.
 Since $H^{2}(S,\mathbf{Z})$ is a free $\mathbf{Z}$-module, we have
 $
 \delta_{11}=\rho^{*}c_{1}(\mathcal{F}).
 $
\end{proof}
\begin{lemma}\label{120819_5Nov16}
 $H^{2}(Y,\mathbf{Z})$ is a free $\mathbf{Z}$-module, and
 $$
 \delta_{1},\dots,\delta_{6},\
 \delta_{7},\dots,\delta_{11},\
 2[E_{12}],\
 2[E_{13}],\
 2[E_{14}]
 $$
 form a $\mathbf{Z}$-basis of the image of the injective homomorphism
 $$
 \rho^{*}:\
 H^{2}(Y,\mathbf{Z})\longrightarrow{H^{2}(S,\mathbf{Z})}
 =H^{2}(C^{(2)},\mathbf{Z})\oplus\bigoplus_{1\leq{i}<{j}\leq{4}}
 \mathbf{Z}[E_{ij}].
 $$
\end{lemma}
\begin{proof}
 By \cite[Corollary~(1.48)]{CZ}, the $\mathbf{Z}$-module $H^{2}(Y,\mathbf{Z})$ does not have a non-trivial torsion element, hence $\rho^{*}:H^{2}(Y,\mathbf{Z})\rightarrow{H^{2}(S,\mathbf{Z})}$ is injective.
 Let
 $H\subset{H^{2}(S,\mathbf{Z})^{\sigma}}
 ={H^{2}(C^{(2)},\mathbf{Z})^{\sigma}\oplus\bigoplus_{1\leq{i}<{j}\leq{4}}}
 \mathbf{Z}[E_{ij}]$
 be the $\mathbf{Z}$-submodule defined by
 $$
 H=\{\gamma\in{H^{2}(S,\mathbf{Z})^{\sigma}}\mid
 \langle{\delta_{i}},\gamma\rangle_{S}\in2\mathbf{Z}\
 \text{for $i=7,8,9,10,11$}\}.
 $$
 Then by Lemma~\ref{123002_3Nov16}, we can show that
 $$
 \delta_{1},\dots,\delta_{6},\
 \delta_{7},\dots,\delta_{11},\
 2[E_{12}],\
 2[E_{13}],\
 2[E_{14}]
 $$
 form a $\mathbf{Z}$-basis of $H$.
 Since $\langle\rho^{*}x,\rho^{*}y\rangle_{S}=2\langle{x},y\rangle_{Y}$
 for $x,y\in{H^{2}(Y,\mathbf{Z})}$,
 by Lemma~\ref{175822_2Nov16} we have $\rho^{*}H^{2}(Y,\mathbf{Z})\subset{H}$.
 By Lemma~\ref{123002_3Nov16}, we can compute the determinant of the symmetric form $\langle\ ,\ \rangle_{S}$ on $H$.
 Since $H$ has the same rank and the same determinant as $\rho^{*}H^{2}(Y,\mathbf{Z})$, we have $\rho^{*}H^{2}(Y,\mathbf{Z})={H}$.
\end{proof}
\begin{remark}\label{145227_5Nov16}
 By Lemma~\ref{120819_5Nov16}, we define a $\mathbf{Z}$-basis
 $\gamma_{1},\dots,\gamma_{14}$ of $H^{2}(Y,\mathbf{Z})$ by
 $$
 \begin{cases}
  \rho^{*}\gamma_{1}=\delta_{1},\\
  \qquad\cdots\\
  \rho^{*}\gamma_{11}=\delta_{11},\\
 \end{cases}
 \qquad
 \begin{cases}
  \rho^{*}\gamma_{12}=2[E_{12}],\\
  \rho^{*}\gamma_{13}=2[E_{13}],\\
  \rho^{*}\gamma_{14}=2[E_{14}].
 \end{cases}
 $$
 Then the intersection matrix is
 $$
 \left(
 \langle\gamma_{i},\gamma_{j}\rangle_{Y}
 \right)%_{1\leq{i}\leq{14},\ 1\leq{j}\leq{14}}
 ={\scriptscriptstyle
 \left(
 \begin{array}{cccccccccccccc}
  {-1}&0&0&0&0&0&0&0&0&0&0&0&0&0\\
  0&{-1}&0&0&0&0&0&0&0&0&1&0&0&0\\
  0&0&0&{-1}&0&0&0&0&0&0&0&0&0&0\\
  0&0&{-1}&0&0&0&0&0&0&0&0&0&0&0\\
  0&0&0&0&0&{-1}&0&0&0&0&0&0&0&0\\
  0&0&0&0&{-1}&0&0&0&0&0&0&0&0&0\\
  0&0&0&0&0&0&{-1}&0&0&0&2&1&1&1\\
  0&0&0&0&0&0&0&{-1}&0&0&2&1&0&0\\
  0&0&0&0&0&0&0&0&{-1}&0&2&0&1&0\\
  0&0&0&0&0&0&0&0&0&{-1}&2&0&0&1\\
  0&1&0&0&0&0&2&2&2&2&{-4}&{-1}&{-1}&{-1}\\
  0&0&0&0&0&0&1&1&0&0&{-1}&{-2}&0&0\\
  0&0&0&0&0&0&1&0&1&0&{-1}&0&{-2}&0\\
  0&0&0&0&0&0&1&0&0&1&{-1}&0&0&{-2}\\
 \end{array}
 \right)},
 $$
 and the classes of curves on $Y$ are
 $$
 \begin{cases}
  [K_{Y}]=\gamma_{1}+\gamma_{2}+2\gamma_{7}+\gamma_{12}+\gamma_{13}+\gamma_{14},\\
  [\rho(\Delta_{\sigma})]=-3\gamma_{7}+\gamma_{8}+\gamma_{9}+\gamma_{10}
  +2\gamma_{11}-2\gamma_{12}-2\gamma_{13}-2\gamma_{14},\\
 \end{cases}
 $$
 $$
 \begin{cases}
  [\rho(\tilde{\Gamma}_{1})]=\gamma_{7},\\
  [\rho(\tilde{\Gamma}_{2})]=\gamma_{8},\\
  [\rho(\tilde{\Gamma}_{3})]=\gamma_{9},\\
  [\rho(\tilde{\Gamma}_{4})]=\gamma_{10},\\
 \end{cases}\quad
 \begin{cases}
  [\rho(E_{12})]=\gamma_{12},\\
  [\rho(E_{13})]=\gamma_{13},\\
  [\rho(E_{14})]=\gamma_{14},\\
  [\rho(E_{34})]=\gamma_{7}+\gamma_{8}-\gamma_{9}-\gamma_{10}+\gamma_{12},\\
  [\rho(E_{24})]=\gamma_{7}-\gamma_{8}+\gamma_{9}-\gamma_{10}+\gamma_{13},\\
  [\rho(E_{23})]=\gamma_{7}-\gamma_{8}-\gamma_{9}+\gamma_{10}+\gamma_{14}.
 \end{cases}
 $$
\end{remark}
\subsection{Construction of $H^{2}(A,\mathbf{Z})$ from $H^{2}(Y,\mathbf{Z})$}\label{154259_6Dec16}
 We set
 $\widetilde{\gamma}=\frac{1}{2}([\rho(\tilde{\Gamma}_{1})]+\dots+[\rho(\tilde{\Gamma}_{4})])
 \in{H^{2}(Y,\mathbf{Q})}$,
 and we define a $\mathbf{Z}$-submodule of ${H^{2}(Y,\mathbf{Q})}$ by
 $$
 H_{Y}=H^{2}(Y,\mathbf{Z})+\mathbf{Z}
 \widetilde{\gamma}
 \subset{H^{2}(Y,\mathbf{Q})}.
 $$
 Since
 $\rho^{*}(\widetilde{\gamma})
 =\frac{1}{2}(\delta_{7}+\delta_{8}+\delta_{9}+\delta_{10})
 =2\widehat{\delta}-\sum_{1\leq{i}<{j}\leq{4}}[E_{ij}]
 \in{H^{2}(S,\mathbf{Z})}$,
 the image of $H_{Y}$ by the pull-back $\rho^{*}:H^{2}(Y,\mathbf{Q})\rightarrow{H^{2}(S,\mathbf{Q})}$ is contained in the integral cohomology $H^{2}(S,\mathbf{Z})$.
 We define $H'_{Y}\subset{H_{Y}}$ as the orthogonal subspace to the classes
 $$
 [\rho{(E_{ij})}]\ (1\leq{i}<{j}\leq{4}),\
 [\rho(\Delta_{\sigma})],\
 [K_{Y}]-\widetilde{\gamma}\
 \in{H_{Y}}.
 $$
 Then the class
 $c_{Y}=[K_{Y}+\rho(\Delta_{\sigma})]$
 is contained in $H'_{Y}$, and the symmetric form
 $$
 \langle\ ,\ \rangle_{Y}:\
 {H^{2}(Y,\mathbf{Q})}\times
 {H^{2}(Y,\mathbf{Q})}\longrightarrow\mathbf{Q}
 $$
 has integral values on $H'_{Y}$.
 We remark that the data $(H'_{Y},\langle\ ,\ \rangle_{Y},c_{Y})$ is defined from the data $(H^{2}(Y,\mathbf{Z}),\langle\ ,\ \rangle_{Y},\Pi_{Y})$, because we can divide the elements of $\Pi_{Y}$ into $4$ type by their self-intersection numbers
 $$
 \begin{cases}
  (K_{Y})^{2}=0,\\
  (\rho(\Delta_{\sigma}))^{2}=-4,\\
  (\rho(\tilde{\Gamma}_{i}))^{2}=-1&(1\leq{i}\leq{4}),\\
  (\rho{(E_{ij})})^{2}=-2&(1\leq{i}<{j}\leq{4}).
 \end{cases}
 $$
\begin{lemma}\label{161526_8Dec16}
 There is an isomorphism $\Psi_{Y}:(H'_{Y},\langle\ ,\ \rangle_{Y},c_{Y})\overset{\simeq}{\rightarrow}(H^{2}(A,\mathbf{Z}),\langle\ ,\ \rangle_{A},[\iota(C)])$ of Hodge structure which preserve the symmetric forms and the ample classes.
\end{lemma}
\begin{proof}
 By Remark~\ref{145227_5Nov16}, we can compute that
 $$
 \gamma_{1}+\widetilde{\gamma}+\gamma_{11},\quad
 \gamma_{2}-3\widetilde{\gamma}+2\gamma_{8}+2\gamma_{9}+2\gamma_{10}
 +\gamma_{11}-\gamma_{12}-\gamma_{13}-\gamma_{14},\quad
 \gamma_{3},\
 \gamma_{4},\
 \gamma_{5},\
 \gamma_{6}
 \in{H'_{Y}}
 $$
 form a $\mathbf{Z}$-basis of $H'_{Y}$.
 Since $H^{1}(A,\mathbf{Z})$ is identified with the kernel of the Gysin homomorphism
 $\pi_{*}:H^{1}(C,\mathbf{Z})\rightarrow{H^{1}(E,\mathbf{Z})}$, we have
 $$
 H^{1}(A,\mathbf{Z})\simeq
 \mathbf{Z}\alpha_{1}\oplus
 \mathbf{Z}(\alpha_{2}-\alpha_{3})\oplus
 \mathbf{Z}\beta_{1}\oplus
 \mathbf{Z}(\beta_{2}-\beta_{3})
 \subset{H^{1}(C,\mathbf{Z})}.
 $$
 By the covering $\lambda:{S}\rightarrow{C^{(2)}}\rightarrow{A}$ which is given in Lemma~\ref{115601_2Dec16}, the pull-back
 $\lambda^{*}:H^{1}(A,\mathbf{Z})\rightarrow{H^{1}(S,\mathbf{Z})}$ coincides with the composition
 $$
 H^{1}(A,\mathbf{Z})\hookrightarrow
 H^{1}(C,\mathbf{Z})\overset{\lambda_{1}}{\longrightarrow}
 H^{1}(C^{(2)},\mathbf{Z})\hookrightarrow
 H^{1}(S,\mathbf{Z}),
 $$
 hence the pull-back
 $\lambda^{*}:H^{2}(A,\mathbf{Z})\rightarrow{H^{2}(S,\mathbf{Z})}$ coincides with the composition
 $$
 H^{2}(A,\mathbf{Z})\simeq
 \bigwedge^{2}H^{1}(A,\mathbf{Z})\hookrightarrow
 \bigwedge^{2}H^{1}(C,\mathbf{Z})\overset{\lambda_{1}}{\longrightarrow}
 \bigwedge^{2}H^{1}(C^{(2)},\mathbf{Z})\overset{\cup}{\longrightarrow}
 H^{2}(C^{(2)},\mathbf{Z})\hookrightarrow
 H^{2}(S,\mathbf{Z}).
 $$
 Then we have $\rho^{*}H'_{Y}=\lambda^{*}H^{2}(A,\mathbf{Z})$, because
 $$
 \begin{cases}
  \rho^{*}(\gamma_{1}+\widetilde{\gamma}+\gamma_{11})
  =(\delta_{0}+\widehat{\delta})+(\delta_{1}-\delta_{2}+2\widehat{\delta})
  =\lambda_{1}(\alpha_{1})\cup\lambda_{1}(\beta_{1})
  +\lambda_{1}(\alpha_{2}-\alpha_{3})\cup\lambda_{1}(\beta_{2}-\beta_{3}),\\
  \rho^{*}(\gamma_{2}-3\widetilde{\gamma}
  +2\gamma_{8}+2\gamma_{9}+2\gamma_{10}
  +\gamma_{11}-\gamma_{12}-\gamma_{13}-\gamma_{14})
  =\delta_{0}+\widehat{\delta}
  =\lambda_{1}(\alpha_{1})\cup\lambda_{1}(\beta_{1}),\\
  \rho^{*}\gamma_{3}=\delta_{3}=
  \lambda_{1}(\alpha_{1})\cup\lambda_{1}(\alpha_{2}-\alpha_{3}),\\
  \rho^{*}\gamma_{4}=\delta_{4}=
  \lambda_{1}(\beta_{1})\cup\lambda_{1}(\beta_{2}-\beta_{3})),\\
  \rho^{*}\gamma_{5}=\delta_{5}=
  \lambda_{1}(\alpha_{1})\cup\lambda_{1}(\beta_{2}-\beta_{3})),\\
  \rho^{*}\gamma_{6}=\delta_{6}=
  \lambda_{1}(\alpha_{2}-\alpha_{3})\cup\lambda_{1}(\beta_{1}).
 \end{cases}
 $$
 Since $\rho^{*}$ and $\lambda^{*}$ are injective homomorphisms of Hodge structures, we have the isomorphism of Hodge structures by
 $\Psi_{Y}=\lambda^{*}\circ(\rho^{*})^{-1}:
 H'_{Y}\rightarrow{H^{2}(A,\mathbf{Z})}$,
 and it satisfies
 $$
 \langle{\Psi_{Y}(x),\Psi_{Y}(x')}\rangle_{A}
 =2\langle{\rho^{*}x,\rho^{*}x'}\rangle_{S}
 =\langle{x,x'}\rangle_{Y}
 $$
 for $x,x'\in{H'_{Y}}$.
 The class
 $$
 [\iota(C)]=2\alpha_{1}\cup\beta_{1}
 +(\alpha_{2}-\alpha_{3})\cup(\beta_{2}-\beta_{3})
 \in{H^{2}(A,\mathbf{Z})}
 $$
 corresponds to $c_{Y}=[K_{Y}+\rho(\Delta_{\sigma})]\in{H'_{Y}}$, because
 $$
 2\lambda_{1}(\alpha_{1})\cup\lambda_{1}(\beta_{1})
 +\lambda_{1}(\alpha_{2}-\alpha_{3})\cup\lambda_{1}(\beta_{2}-\beta_{3})
 =2(\delta_{0}+\widehat{\delta})
 +(\delta_{1}-\delta_{2}+2\widehat{\delta})
 =\rho^{*}[K_{Y}+\rho(\Delta_{\sigma})].
 $$
\end{proof}
\subsection{Construction of $H^{2}(Y,\mathbf{Z})$ from $H^{2}(A,\mathbf{Z})$}\label{154538_6Dec16}
Let $H=\bigoplus_{i=1}^{8}\mathbf{Z}{v}_{i}$ be the lattice defined by
$$
\langle{v}_{i},{v}_{j}\rangle_{A}=
\begin{cases}
 -1&(i=j)\\
 0&(i\neq{j}).
\end{cases}
$$
We set
$$
\widetilde{v}=\frac{1}{2}
([\iota(C)]+{v}_{1}+{v}_{2}+{v}_{3}-{v}_{4}-{v}_{5}-{v}_{6}-{v}_{7}-{v}_{8})
\in{H^{2}(A,\mathbf{Q})\oplus{(\mathbf{Q}\otimes_{\mathbf{Z}}H)}},
$$
and we define a $\mathbf{Z}$-submodule of
$H^{2}(A,\mathbf{Q})\oplus{(\mathbf{Q}\otimes_{\mathbf{Z}}H)}$
by
$$
H_{A}=(H^{2}(A,\mathbf{Z})\oplus{H})+\mathbf{Z}\widetilde{v}.
$$
Let $H'_{A}\subset{H_{A}}$ be the $\mathbf{Z}$-submodule defined by
$$
H'_{A}=\{x\in{H_{A}}\mid\langle{x},\widetilde{v}\rangle_{A}\in\mathbf{Z}\}.
$$
Then the symmetric form
$$
\langle\ ,\ \rangle_{A}:\
(H^{2}(A,\mathbf{Q})\oplus{(\mathbf{Q}\otimes_{\mathbf{Z}}H)})\times
(H^{2}(A,\mathbf{Q})\oplus{(\mathbf{Q}\otimes_{\mathbf{Z}}H)})
\longrightarrow
\mathbf{Q}
$$
has integral values on $H'_{A}$, and the finite subset
$$
\Pi_{A}=\{
[\iota(C)]-2{v}_{7},\
2{v}_{7},\
\widetilde{v},\
\widetilde{v}-{v}_{2}-{v}_{3},\
\widetilde{v}-{v}_{1}-{v}_{3},\
\widetilde{v}-{v}_{1}-{v}_{2},\
\pm{v}_{1}+{v}_{4},\
\pm{v}_{2}+{v}_{5},\
\pm{v}_{3}+{v}_{6}
\}
$$
is contained in $H'_{A}$.
\begin{lemma}\label{132213_8Dec16}
 There is an isomorphism $\Psi_{A}:{(H'_{A},\langle\ ,\ \rangle_{A},\Pi_{A})}\overset{\simeq}{\rightarrow}(H^{2}(Y,\mathbf{Z}),\langle\ ,\ \rangle_{Y},\Pi_{Y})$ of Hodge structure which preserve the symmetric forms and the finite subsets.
\end{lemma}
\begin{proof}
 We extend the homomorphism $\lambda^{*}:H^{2}(A,\mathbf{Z})\hookrightarrow{H^{2}(S,\mathbf{Z})}$ to $\lambda^{*}:H_{A}\hookrightarrow{H^{2}(S,\mathbf{Z})}$
 by
 $$
 \begin{cases}
  \lambda^{*}\widetilde{v}=\widehat{\delta}-[E_{12}]-[E_{13}]-[E_{14}],\\
  \lambda^{*}v_{1}=[E_{34}]-[E_{12}],\\
  \lambda^{*}v_{2}=[E_{24}]-[E_{13}],\\
  \lambda^{*}v_{3}=[E_{23}]-[E_{14}],\\
 \end{cases}\qquad
 \begin{cases}
  \lambda^{*}v_{4}=[E_{34}]+[E_{12}],\\
  \lambda^{*}v_{5}=[E_{24}]+[E_{13}],\\
  \lambda^{*}v_{6}=[E_{23}]+[E_{14}],\\
  \lambda^{*}v_{7}=[\Delta_{\sigma}],\\
 \end{cases}
 $$
 where we remark that the image $\lambda^{*}v_{8}$ is determined by the relation
 $v_{8}=2\widetilde{v}-[\iota(C)]-v_{1}-v_{2}-v_{3}+v_{4}+v_{5}+v_{6}+v_{7}$, and it satisfies
 $$
 \langle\lambda^{*}x,\lambda^{*}x'\rangle_{A}=2\langle{x,x'}\rangle_{S}
 $$
 for $x,x'\in{H_{A}}$.
 Since $\lambda^{*}H^{2}(A,\mathbf{Z})=\rho^{*}H'_{Y}\subset\rho^{*}H_{Y}$
 and
 $$
 \begin{cases}
  \lambda^{*}\widetilde{v}=\rho^{*}\gamma_{7},\\
  \lambda^{*}v_{1}=\rho^{*}(\widetilde{\gamma}-\gamma_{9}-\gamma_{10}),\\
  \lambda^{*}v_{2}=\rho^{*}(\widetilde{\gamma}-\gamma_{8}-\gamma_{10}),\\
  \lambda^{*}v_{3}=\rho^{*}(\widetilde{\gamma}-\gamma_{8}-\gamma_{9}),\\
 \end{cases}\qquad
 \begin{cases}
  \lambda^{*}v_{4}=\rho^{*}(\widetilde{\gamma}-\gamma_{9}-\gamma_{10}+\gamma_{12}),\\
  \lambda^{*}v_{5}=\rho^{*}(\widetilde{\gamma}-\gamma_{8}-\gamma_{10}+\gamma_{13}),\\
  \lambda^{*}v_{6}=\rho^{*}(\widetilde{\gamma}-\gamma_{8}-\gamma_{9}+\gamma_{14}),\\
  \lambda^{*}v_{7}=\rho^{*}(-3\widetilde{\gamma}
  +2\gamma_{8}+2\gamma_{9}+2\gamma_{10}
 +\gamma_{11}-\gamma_{12}-\gamma_{13}-\gamma_{14}),\\
 \end{cases}
 $$
 we have $\lambda^{*}H_{A}\subset\rho^{*}{H_{Y}}$.
 Since $\lambda^{*}H_{A}$ and $\rho^{*}{H_{Y}}$ have the same determinants by $\langle\ ,\ \rangle_{S}$, we have $\lambda^{*}H_{A}=\rho^{*}{H_{Y}}$, and
 $$
 \lambda^{*}H'_{A}=
 \{\lambda^{*}x\in\lambda^{*}H_{A}\mid\langle{\lambda^{*}x,\lambda^{*}\widetilde{v}}\rangle_{S}
 \in2\mathbf{Z}\}
 =\{\rho^{*}y\in\rho^{*}H_{Y}\mid\langle{\rho^{*}y,\rho^{*}\gamma_{7}}\rangle_{S}
 \in2\mathbf{Z}\}
 =\rho^{*}H^{2}(Y,\mathbf{Z}).
 $$
 The finite set $\Pi_{Y}$ corresponds to $\Pi_{A}$, because
 $$
 \begin{cases}
  \rho^{*}[K_{Y}]=\lambda^{*}([\iota(C)]-2{v}_{7}),\\
  \rho^{*}[\rho(\Delta_{\sigma})]=\lambda^{*}(2{v}_{7}),\\
  \rho^{*}[\rho(\tilde{\Gamma}_{1})]=\lambda^{*}\widetilde{v},\\
  \rho^{*}[\rho(\tilde{\Gamma}_{2})]=\lambda^{*}(\widetilde{v}-{v}_{2}-{v}_{3}),\\
  \rho^{*}[\rho(\tilde{\Gamma}_{3})]=\lambda^{*}(\widetilde{v}-{v}_{1}-{v}_{3}),\\
  \rho^{*}[\rho(\tilde{\Gamma}_{4})]=\lambda^{*}(\widetilde{v}-{v}_{1}-{v}_{2}),\\
 \end{cases}\qquad
 \begin{cases}
  \rho^{*}[\rho(E_{12})]=\lambda^{*}(-{v}_{1}+{v}_{4}),\\
  \rho^{*}[\rho(E_{13})]=\lambda^{*}(-{v}_{2}+{v}_{5}),\\
  \rho^{*}[\rho(E_{14})]=\lambda^{*}(-{v}_{3}+{v}_{6}),\\
  \rho^{*}[\rho(E_{34})]=\lambda^{*}({v}_{1}+{v}_{4}),\\
  \rho^{*}[\rho(E_{24})]=\lambda^{*}({v}_{2}+{v}_{5}),\\
  \rho^{*}[\rho(E_{34})]=\lambda^{*}({v}_{3}+{v}_{6}).
 \end{cases}
 $$
\end{proof}
\begin{corollary}
 The transcendental lattice of $H^{2}(Y,\mathbf{Z})$ is isomorphic to the transcendental lattice of $H^{2}(A,\mathbf{Z})$.
\end{corollary}
\begin{proof}
 The transcendental lattice of $H^{2}(Y,\mathbf{Z})$ is defined as
 $$
 H^{2}(Y,\mathbf{Z})_{\mathrm{tr}}=
 \{x\in{H^{2}(Y,\mathbf{Z})}\mid\langle{x,y}\rangle_{Y}=0\
 \text{for}\ y\in\NS{(Y)}\}.
 $$
 Since
 $2(H_{Y}\cap{(\mathbf{C}\otimes_{\mathbf{Z}}H_{Y})^{1,1}})
 \subset\NS{(Y)}$,
 the transcendental lattice $H^{2}(Y,\mathbf{Z})_{\mathrm{tr}}$ is contained in
 $$
 H_{Y,\mathrm{tr}}=\{x\in{H_{Y}}\mid\langle{x,y}\rangle_{Y}=0\
 \text{for}\ y\in{H_{Y}\cap{(\mathbf{C}\otimes_{\mathbf{Z}}H_{Y})^{1,1}}}\}.
 $$
 For $x\in{H^{2}(Y,\mathbf{Z})}$ and ${m}\in\mathbf{Z}$, if
 $x+m\widetilde{\gamma}\in{H_{Y,\mathrm{tr}}}$, then
 $\langle{x+m\widetilde{\gamma},\gamma_{7}}\rangle_{Y}=0$,
 hence we have
 $\langle{x,\gamma_{7}}\rangle_{Y}=\frac{1}{2}m\in\mathbf{Z}$
 and
 $x+m\widetilde{\gamma}\in{H^{2}(Y,\mathbf{Z})}$.
 Since $\NS{(Y)}\subset{H_{Y}\cap{(\mathbf{C}\otimes_{\mathbf{Z}}H_{Y})^{1,1}}}$, we have $H^{2}(Y,\mathbf{Z})_{\mathrm{tr}}=H_{Y,\mathrm{tr}}$.
 In the similar way, we can show that
 $$
 H^{2}(A,\mathbf{Z})_{\mathrm{tr}}
 =(H^{2}(A,\mathbf{Z})\oplus{H})_{\mathrm{tr}}
 =H_{A,\mathrm{tr}}.
 $$
 By the proof of Lemma~\ref{132213_8Dec16}, there is an isomorphism of Hodge structures $(H_{Y},\langle\ ,\ \rangle_{Y})\simeq(H_{A},\langle\ ,\ \rangle_{A})$, hence we have $H_{Y,\mathrm{tr}}\simeq{H_{A,\mathrm{tr}}}$.
\end{proof}
\begin{remark}
 Both the Hodge structures $H^{2}(Y,\mathbf{Z})$ and
 $H^{2}(A,\mathbf{Z})\oplus{H}$
 are sublattices of index $2$ in $H_{Y}\simeq{H_{A}}$.
 But $H^{2}(Y,\mathbf{Z})$ is not isometric to $H^{2}(A,\mathbf{Z})\oplus{H}$.
 In generic case, we can compute that the N\'{e}ron-Severi lattice of $Y$ is
 $\NS{(Y)}\simeq\mathbf{1}\oplus(\mathbf{-1})^{\oplus5}\oplus(-A_{3})$,
 where $(-A_{3})$ denotes the negative definite root lattice of type $A_{3}$.
 It is not isomorphic to
 $\NS{(A)}\oplus{H}\simeq\mathbf{4}\oplus(\mathbf{-1})^{\oplus8}$.
\end{remark}
\section{Torelli problem}\label{183635_16Oct16}
\subsection{Infinitesimal Torelli problem}
Let $Y=Y(C/E)$ be the surface constructed from a bielliptic curve $\pi:C\rightarrow{E}$.
\begin{proposition}\label{142139_12Oct16}
 The infinitesimal period map
 $$
 H^{1}(Y,{T}_{Y})\longrightarrow
 \Hom{(H^{0}(Y,\Omega_{Y}^{2}),H^{1}(Y,\Omega_{Y}^{1}))}
 $$
 is not injective.
\end{proposition}
\begin{proof}
 By the duality, we prove that the cup product homomorphism
 $$
 \mu:\
 H^{0}(Y,\Omega_{Y}^{2})\otimes
 H^{1}(Y,\Omega_{Y}^{1})\longrightarrow
 H^{1}(Y,\Omega_{Y}^{2}\otimes\Omega_{Y}^{1})
 $$
 is not surjective.
 We have the exact sequence
 $$
 0\longrightarrow\Omega_{Y}^{1}
 \overset{s}{\longrightarrow}\Omega_{Y}^{2}\otimes\Omega_{Y}^{1}
 \longrightarrow
 (\Omega_{Y}^{2}\otimes\Omega_{Y}^{1})\vert_{K}
 \longrightarrow0,
 $$
 where $K$ denotes the zeros of a nontrivial section $s\in{H^{0}(Y,\Omega_{Y}^{2})}$.
 Since $h^{0}(Y,\Omega_{Y}^{2})=1$, the image of $\mu$ coincides with the image of
 $$
 H^{1}(Y,\Omega_{Y}^{1})\overset{s}{\longrightarrow}
 H^{1}(Y,\Omega_{Y}^{2}\otimes\Omega_{Y}^{1}).
 $$
 Since $h^{2}(Y,\Omega_{Y}^{1})=1$, it is enough to show that
 $h^{1}(K,(\Omega_{Y}^{2}\otimes\Omega_{Y}^{1})\vert_{K})\geq2$.
 By Proposition~\ref{161646_15Oct16}, $K$ is the fiber of $f:Y\rightarrow\Pic^{(2)}{(E)}$ at $\eta\in\Pic^{(2)}{(E)}$.
 If $C$ is not a hyperelliptic curve, then by Corollary~\ref{112709_18Oct16}, we have the exact sequence
 $$
 0\longrightarrow
 H^{1}(K,{T}_{K})\longrightarrow
 H^{1}(K,{T}_{Y}\vert_{K})\longrightarrow
 H^{1}(K,\mathcal{O}_{K})\otimes{T}_{E}\vert_{\eta}\longrightarrow0.
 $$
 Since
 $\Omega_{Y}^{2}\vert_{K}\simeq\Omega_{K}^{1}\otimes({T}_{E}\vert_{\eta})^{\vee}
 \simeq\mathcal{O}_{K}$,
 we have
 $$
 h^{1}(K,(\Omega_{Y}^{2}\otimes\Omega_{Y}^{1})\vert_{K})
 =h^{1}(K,(\Omega_{Y}^{2}\otimes\Omega_{Y}^{2}\otimes{T}_{Y})\vert_{K})
 =h^{1}(K,{T}_{Y}\vert_{K})
 =2.
 $$
 If $C$ is a hyperelliptic curve, then by Proposition~\ref{160302_14Oct16} we denote
 $K=\bigcup_{i=1}^{4}K_{i}$, where $K_{i}$ is a $(-2)$-curve on $Y$.
 Then by Lemma~\ref{165642_21Oct16}, we have
 $$
 \begin{array}{rcl}
  h^{0}(K,(\Omega_{Y}^{2}\otimes\Omega_{Y}^{1})\vert_{K})&\geq
   &h^{0}(K,\Omega_{Y}^{1}(K-K_{1})\vert_{K_{2}\cup{K_{3}}\cup{K_{4}}})
   +h^{0}(K,\Omega_{Y}^{1}(K_{1})\vert_{K_{1}})\\
  &\geq
   &h^{0}(K,\Omega_{Y}^{1}(K-K_{1}-K_{2})\vert_{{K_{3}}\cup{K_{4}}})
   +\sum_{i=1}^{2}h^{0}(K,\Omega_{Y}^{1}(K_{i})\vert_{K_{i}})\\
  &\geq
   &\sum_{i=1}^{4}h^{0}(K,\Omega_{Y}^{1}(K_{i})\vert_{K_{i}}).\\
 \end{array}
 $$
 Since $K_{i}$ is a $(-2)$-curve on $Y$,
 by the exact sequence
 $$
 0\longrightarrow\mathcal{N}_{K_{i}/Y}^{\vee}
 \longrightarrow\Omega_{Y}^{1}\vert_{K_{i}}
 \longrightarrow\Omega_{K_{i}}^{1}
 \longrightarrow0,
 $$
 we have a exact sequence
 $$
 0\longrightarrow\mathcal{O}_{K_{i}}
 \longrightarrow\Omega_{Y}^{1}(K_{i})\vert_{K_{i}}
 \longrightarrow\mathcal{O}_{K_{i}}(-4)
 \longrightarrow0.
 $$
 Hence we have $h^{0}(K,\Omega_{Y}^{1}(K_{i})\vert_{K_{i}})=1$.
 By the Serre duality $h^{i}(Y,\Omega_{Y}^{1}(K))=h^{2-i}(Y,\Omega_{Y}^{1}(-K))$, we have
 $$
 \chi(K,(\Omega_{Y}^{2}\otimes\Omega_{Y}^{1})\vert_{K})
 =\chi(Y,\Omega_{Y}^{2}\otimes\Omega_{Y}^{1})-
 \chi(Y,\Omega_{Y}^{1})
 =\chi(Y,\Omega_{Y}^{1}(-K))-
 \chi(Y,\Omega_{Y}^{1})
 =-\chi(K,\Omega_{Y}^{1}\vert_{K}),
 $$
 and by $\Omega_{Y}^{2}\vert_{K}\simeq\mathcal{O}_{K}$, we have
 $\chi(K,(\Omega_{Y}^{2}\otimes\Omega_{Y}^{1})\vert_{K})
 =\chi(K,\Omega_{Y}^{1}\vert_{K})$.
 Hence we have $\chi(K,(\Omega_{Y}^{2}\otimes\Omega_{Y}^{1})\vert_{K})=0$ and
 $$
 h^{1}(K,(\Omega_{Y}^{2}\otimes\Omega_{Y}^{1})\vert_{K})=
 h^{0}(K,(\Omega_{Y}^{2}\otimes\Omega_{Y}^{1})\vert_{K})\geq4
 $$
\end{proof}
\begin{remark}
 $h^{0}(Y,{T}_{Y})=0$, $h^{1}(Y,{T}_{Y})=11$ and
 $h^{2}(Y,{T}_{Y})=1$.
\end{remark}
\begin{remark}
 If $C$ is not a hyperelliptic curve, then the kernel of the infinitesimal period map is of dimension $1$.
 If $C$ is a hyperelliptic curve, then the dimension of the kernel is grater than $2$.
\end{remark}
\begin{lemma}\label{165642_21Oct16}
 Let $C_{1},C_{2}$ be curves on nonsingular surface $Y$, and let $\mathcal{F}$ be a locally free sheaf of $\mathcal{O}_{Y}$-modules.
 If $C_{1}$ and $C_{2}$ have no common components, then
 $$
 h^{0}(C_{1}\cup{C_{2}},\mathcal{F}\vert_{C_{1}\cup{C_{2}}})
 \geq
 h^{0}(C_{1},\mathcal{F}(-C_{2})\vert_{C_{1}})+
 h^{0}(C_{2},\mathcal{F}(-C_{1})\vert_{C_{2}}).
 $$
\end{lemma}
\begin{proof}
 By the commutative diagram
 $$
 \begin{array}{ccccccccc}
  &&0&&0&&&&\\
  &&\downarrow&&\downarrow&&&&\\
  &&\mathcal{F}(-C_{1}-C_{2})
   &=&\mathcal{F}(-C_{1}-C_{2})
   &&&&\\
  &&\downarrow&&\downarrow&&&&\\
  0&\longrightarrow&\mathcal{F}(-C_{1})
   &\longrightarrow&\mathcal{F}
   &\longrightarrow&\mathcal{F}\vert_{C_{1}}
   &\longrightarrow&0\\
  &&\downarrow&&\downarrow&&&&\\
  &&\mathcal{F}(-C_{1})\vert_{C_{2}}
   &\longrightarrow&\mathcal{F}\vert_{C_{1}\cup{C_{2}}}
   &&
   &&\\
  &&\downarrow&&\downarrow&&&&\\
  &&0&&0&&&&\\
 \end{array}
 $$
 of the exact sequences of $\mathcal{O}_{Y}$-modules, we have the exact sequence
 $$
 0\longrightarrow
 \mathcal{F}(-C_{1})\vert_{C_{2}}
 \longrightarrow\mathcal{F}\vert_{C_{1}\cup{C_{2}}}
 \longrightarrow\mathcal{F}\vert_{C_{1}}
 \longrightarrow0.
 $$
 By the commutative diagram
 $$
 \begin{array}{ccccccc}
  & & & &0& &0\\
  & & & &\downarrow& &\downarrow\\
  & & & &H^{0}(C_{1},\mathcal{F}(-C_{2})\vert_{C_{1}})
   &=&H^{0}(C_{1},\mathcal{F}(-C_{2})\vert_{C_{1}})\\
  & & & &\downarrow& &\downarrow\\
  0&\longrightarrow&H^{0}(C_{2},\mathcal{F}(-C_{1})\vert_{C_{2}})
   &\longrightarrow&H^{0}(C_{1}\cup{C_{2}},\mathcal{F}\vert_{C_{1}\cup{C_{2}}})
   &\longrightarrow&H^{0}(C_{1},\mathcal{F}\vert_{C_{1}})\\
 \end{array}
 $$
 of exact sequences, we have a injective homomorphism
 $$
 H^{0}(C_{1},\mathcal{F}(-C_{2})\vert_{C_{1}})\oplus
 H^{0}(C_{2},\mathcal{F}(-C_{1})\vert_{C_{2}})\longrightarrow
 H^{0}(C_{1}\cup{C_{2}},\mathcal{F}\vert_{C_{1}\cup{C_{2}}}).
 $$
\end{proof}
\subsection{Global Torelli problem}
Let $\mathcal{N}$ be the set of isomorphism class of surfaces which have the same topological type as the elliptic surface $Y$.
By Proposition~\ref{112036_17Oct16}, the moduli space $\mathcal{M}$ of bielliptic curves of genus $3$ is embedded in $\mathcal{N}$, hence we regard as $\mathcal{M}\subset\mathcal{N}$.
Let $\mathcal{M}_{0}$ be the inverse image of the self-dual locus of $\mathcal{A}_{2}^{(1,2)}$ by the Prym map $\mathcal{M}\rightarrow\mathcal{A}_{2}^{(1,2)}$.
\begin{theorem}\label{181010_13Dec16}
 For $[Y]\in\mathcal{M}\subset\mathcal{N}$, the locus
 $$
 \mathcal{M}(H^{2}(Y))=\{[Y']\in\mathcal{M}\mid
 (H^{2}(Y,\mathbf{Z}),\langle\ ,\ \rangle_{Y},\Pi_{Y})
 \simeq
 (H^{2}(Y',\mathbf{Z}),\langle\ ,\ \rangle_{Y'},\Pi_{Y'})\}
 $$
 is $1$-dimensional.
 If $[Y]\in\mathcal{M}_{0}$, then $\mathcal{M}(H^{2}(Y))$ is a rational curve.
 If $[Y]\notin\mathcal{M}_{0}$, then $\mathcal{M}(H^{2}(Y))$ is a union of two rational curves.
\end{theorem}
\begin{proof}
 We fix $[Y]=[Y(C/E)]\in\mathcal{M}\subset\mathcal{N}$.
 If $[Y']=[Y'(C'/E')]\in\mathcal{M}(H^{2}(Y))$, then by Theorem~\ref{104530_30Nov16}, there is an isomorphism
 $$
 (H^{2}(A,\mathbf{Z}),\langle\ ,\ \rangle_{A},[\iota(C)])
 \simeq(H^{2}(A',\mathbf{Z}),\langle\ ,\ \rangle_{A'},[\iota(C')])
 $$
 of Hodge structures of corresponding $(1,2)$-polarized abelian surfaces.
 By \cite[Theorem~I]{Sh}, the abelian surface $A'$ is isomorphic to $A$ or its dual $P=A^{\vee}$ as complex tori.
 When $A'$ is isomorphic to $A$, then by \cite[Theorem~1]{Sh}, the isomorphism $\Psi:H^{2}(A,\mathbf{Z})\rightarrow{H^{2}(A',\mathbf{Z})}$ is given as $\Psi=\pm\psi^{*}$ by an isomorphism
 $\psi:A'{\rightarrow}{A}$.
 Since $[\iota(C)]$ and $[\iota(C')]$ are ample classes with $\Psi[\iota(C)]=[\iota(C')]$, we have $\Psi=\psi^{*}$, and $\psi$ gives an isomorphism of polarized abelian surfaces.
 When $A'$ is isomorphic to $P$, then by \cite[Theorem~2]{Sh}, the isomorphism $\Psi:H^{2}(A,\mathbf{Z})\rightarrow{H^{2}(A',\mathbf{Z})}$ is given as $\Psi\circ\alpha_{A}=\pm\psi^{*}$ by an isomorphism
 $\psi:A'{\rightarrow}{P}$, where
 $\alpha_{A}:H^{2}(P,\mathbf{Z})\rightarrow{H^{2}(A,\mathbf{Z})}$
 is the Hodge isometry defined in \cite[Lemma~3]{Sh}.
 Since the Hodge isometry $\alpha_{A}$ preserve the class of the $(1,2)$-polarizations, in the same way as the case $A'\simeq{A}$, we have $\Psi\circ\alpha_{A}=\psi^{*}$, and $\psi$ gives an isomorphism of polarized abelian surfaces.
 Hence we have
 $$
 \mathcal{M}(H^{2}(Y))=
 \mathcal{M}(P)\cup\mathcal{M}(A)
 =\mathcal{M}(\Prym{(C/E)})\cup
 \mathcal{M}(\Prym{(C/E)}^{\vee}).
 $$
\end{proof}
\begin{theorem}\label{163939_14Jan17}
 For $[Y]\in\mathcal{M}\subset\mathcal{N}$, the locus
 $$
 \mathcal{M}(H^{2}(Y)\oplus{H^{1}(Y)})\\
 =\{[Y']\in\mathcal{M}\mid{[Y']\in{\mathcal{M}(H^{2}(Y))}}\
 \text{and}\ H^{1}(Y,\mathbf{Z})\simeq{H^{1}(Y',\mathbf{Z})}\}
 $$
 is a finite set.
 For general $[F]\in\mathcal{M}$, it consists of $12$ points.
\end{theorem}
\begin{proof}
 Since
 $H^{1}(Y,\mathbf{Z})\simeq
 H^{1}(E,\mathbf{Z})$
 by Corollary~\ref{143245_20Jan17}, it follows from Theorem~\ref{181010_13Dec16} and Proposition~\ref{175005_11Dec16}.
\end{proof}

\begin{theorem}\label{163320_13Jan17}
 For general $[Y]\in\mathcal{M}\subset\mathcal{N}$, the locus
 \begin{multline*}
  \mathcal{M}(H^{2}(Y)\oplus{H^{1}(Y)}\oplus{H^{1}(K_{Y})})\\
  =\{[Y']\in\mathcal{M}\mid{[Y']\in{\mathcal{M}(H^{2}(Y)\oplus{H^{1}(Y)})}\
  \text{and}\ H^{1}(K_{Y},\mathbf{Z})\simeq{H^{1}(K_{Y'},\mathbf{Z})}}\}
 \end{multline*}
 consists of $1$ point $[Y]$.
\end{theorem}
For $(s_{0},\dots,s_{2},\lambda)\in\mathbf{A}^{4}$, we set the bielliptic curve
$C_{(s,\lambda)}\rightarrow{E_{(s,\lambda)}}$ by
$$
C_{(s,\lambda)}
=\{[x:y:z]\in\mathbf{P}^{2}\mid
(z^{2}+s_{0}x^{2}+s_{1}xy+s_{2}y^{2})^{2}
=x(x-y)(x-\lambda{y})y\},
$$
where the bielliptic involution is given by $z\mapsto{-z}$.
We set the open subset $U\subset\mathbf{A}^{4}$ by
$$
U=\{(s_{0},\dots,s_{2},\lambda)\in\mathbf{A}^{4}\mid
\text{$C_{(s,\lambda)}$ is nonsingular}\}.
$$
Then we have a dominant morphism
$$
U\longrightarrow\mathcal{M};\
(s_{0},\dots,s_{2},\lambda)\longmapsto
[C_{(s,\lambda)}\rightarrow{E_{(s,\lambda)}}].
$$
\begin{lemma}\label{162617_13Jan17}
 For general $(s_{0},\dots,s_{2},\lambda)\in{U}$, the map
 $$
 \mathbf{j}_{(s_{0},\dots,s_{2},\lambda)}:\
 \mathbf{P}^{1}(a\text{-line})\longrightarrow
 \mathbf{P}^{1}\times\mathbf{P}^{1};\
 a\longmapsto(j_{F}(a),j_{K}(a)),
 $$
 is generically injective, where $j_{K}$ and $j_{F}$ are defined for $C_{(s,\lambda)}$ in Remark~\ref{153107_13Jan17} and Remark~\ref{153122_13Jan17}.
\end{lemma}
\begin{proof}
 Let $I$ be the image of
 $$
 \mathbf{P}^{1}(a\text{-line})\longrightarrow
 \mathbf{P}^{1}\times\mathbf{P}^{1};\
 a\longmapsto(j_{F}(a),j_{K}(a)),
 $$
 and let $\tilde{I}$ be the normalization of $I$.
 If the induced morphism $\mathbf{P}^{1}(a\text{-line})\rightarrow\tilde{I}$ is not an isomorphism, then it has ramification points, hence $(s_{0},\dots,s_{2},\lambda)$ is contained in
 $$
 \{(s_{0},\dots,s_{2},\lambda)\in{U}\mid\text{$j_{K}$ and $j_{F}$ has common ramification points on the $a$-line}\}.
 $$
 It is a proper closed subset of $U$ by Example~\ref{175545_13Jan17}.
\end{proof}
\begin{lemma}\label{162638_13Jan17}
 For general $(s_{0},\dots,s_{2},\lambda)\in{U}$, the image of
 $$
 \mathbf{j}_{(s_{0},\dots,s_{2},\lambda)}:\
 \mathbf{P}^{1}(a\text{-line})\longrightarrow
 \mathbf{P}^{1}\times\mathbf{P}^{1};\
 a\longmapsto(j_{F}(a),j_{K}(a)),
 $$
 does not coincides with the image of
 $$
 \check{\mathbf{j}}_{(s_{0},\dots,s_{2},\lambda)}:\
 \mathbf{P}^{1}(a\text{-line})\longrightarrow
 \mathbf{P}^{1}\times\mathbf{P}^{1};\
 a\longmapsto(j_{K}(a),j_{F}(a)).
 $$
\end{lemma}
\begin{proof}
 By Example~\ref{175545_13Jan17}, the locus
 $$
 \{(s_{0},\dots,s_{2},\lambda)\in{U}\mid\Image{(\mathbf{j}_{(s,\lambda)})}=
 \Image{(\check{\mathbf{j}}_{(s,\lambda)})}\}
 $$
 is a proper closed subset of $U$.
\end{proof}
\begin{proof}[Proof of Theorem~\ref{163320_13Jan17}]
 Let $(A,[\mathcal{L}_{A}])\in\mathcal{A}_{2}^{(1,2)}$ be a general point,
 and let $[D\rightarrow{F}]\in\mathcal{M}(A)$ be a point such that
 $D$ is not hyperelliptic.
 Then we may assume that $D$ is defined by
 $$
 D=\{[x:y:z]\in\mathbf{P}^{2}\mid
 z^{4}+2(s_{0}x^{2}+s_{1}xy+s_{2}y^{2})z^{2}+x(x-y)(x-\lambda{y})y=0\}.
 $$
 By Lemma~\ref{095421_30Nov16}, the dual bielliptic curve $D^{\vee}$ of $D$
 is equal to $C_{(s,\lambda)}$.
 By Lemma~\ref{110908_14Dec16}, the bielliptic fibration
 $P^{\sim}\rightarrow{\mathbf{P}^{1}=|\mathcal{L}_{P}}|$
 gives a covering $|\mathcal{L}_{P}|\dashrightarrow\mathcal{M}(A)$ of degree $4$, where $(P,[\mathcal{L}])$ denotes the dual of $(A,[\mathcal{L}_{A}])$.
 By comparing the bielliptic fibration
 $A'=C^{(2)}/(\sigma^{(2)}\circ\kappa)\rightarrow|\mathcal{O}_{B}(2\eta)|
 =\mathbf{P}^{1}(a\text{-line})$
 for $C=C_{(s,\lambda)}$ with
 $P^{\sim}\rightarrow{\mathbf{P}^{1}=|\mathcal{L}_{P}}|$
 in the diagram of Remark~\ref{111720_14Dec16},
 the locus $\mathcal{M}(A)$ is identified with an open subset of the $\mathbf{P}^{1}(a\text{-line})$.
 By Lemma~\ref{162617_13Jan17} and Lemma~\ref{162638_13Jan17},
 $$
 \mathbf{j}_{A}:\
 \mathcal{M}(A)\longrightarrow\mathbf{P}^{1}\times\mathbf{P}^{1};\
 [Y]=[Y(C/E)]\longmapsto(j(E),j(K_{Y}))
 $$
 is generically injective, and
 $$
 \overline{\mathbf{j}_{A}(\mathcal{M}(A))}
 =\Image{(\mathbf{j}_{(s,\lambda)})}\neq
 \Image{(\check{\mathbf{j}}_{(s,\lambda)})}
 =\overline{\mathbf{j}_{A}(\mathcal{M}(P))}.
 $$
 By Theorem~\ref{181010_13Dec16},
 if $\mathbf{j}_{A}([Y])$ is not contained in the singular locus of
 $\Image{(\mathbf{j}_{(s,\lambda)})}\cup
 \Image{(\check{\mathbf{j}}_{(s,\lambda)})}$,
 then the set $\mathcal{M}(H^{2}(Y)\oplus{H^{1}(Y)}\oplus{H^{1}(K_{Y})})$ consists of $1$ point.
\end{proof}

\subsection{Examples}
\begin{example}
 Let $\pi:C\rightarrow{E}$ be the bielliptic curve defined by
 $$
 C=\{[x:y:z]\in\mathbf{C}\mid
 z^{4}=T(x,y)\},
 $$
 which is isomorphic to its dual $C^{\vee}$ by Lemma~\ref{095421_30Nov16}.
 In this locus, the bielliptic curve $C$ is uniquely determined by the base $E$.
 Since $C$ has the automorphism
 $z\mapsto{i}z$,
 we can compute the period matrix of $C$ explicitly.
 Then we have
 $$
 A=\Prym{(C/E)}\simeq\mathbf{C}^{2}/
 \mathbf{Z}(1,0)\oplus
 \mathbf{Z}(0,1)\oplus
 \mathbf{Z}\bigl(\frac{1+i}{2},-\frac{1}{2}\bigr)\oplus
 \mathbf{Z}\bigl(-1,\frac{-1+i}{2}\bigr),
 $$
 and the $(1,2)$-polarization $[\mathcal{L}_{A}]$ is given by the Hermitian form
 $$
 H:\mathbf{C}^{2}\rightarrow\mathbf{C}^{2};\
 H((z_{1},z_{2}),(w_{1},w_{2}))=2(z_{1}\bar{w}_{1}+2z_{2}\bar{w}_{2}).
 $$
 We remark that the polarized abelian surface $(A,[\mathcal{L}_{A}])$ does not depend on $T(x,y)$, and it has the automorphism
 $$
 \mu:A\rightarrow{A};\
 (z_{1},z_{2})\longmapsto
 (z_{1},z_{2})
 \left(
 \begin{array}{cc}
  \frac{1+i}{2}&-\frac{1}{2}\\
  -1&\frac{-1+i}{2}\\
 \end{array}
 \right)
 $$
 of order $12$.
 Since $j_{F}=j_{K}:\mathbf{P}^{1}(a\text{-line})\rightarrow\mathbf{P}^{1}(j\text{-line})$,
 the map
 $\mathbf{j}=(j_{F},j_{K}):\mathbf{P}^{1}(a\text{-line})\rightarrow\mathbf{P}^{1}\times\mathbf{P}^{1}$
 is not generically injective,
 but
 $\mathbf{j}_{A}:\mathcal{M}(A)\subset\mathbf{P}^{1}(j\text{-line})
 \rightarrow\mathbf{P}^{1}\times\mathbf{P}^{1}$
 is injective.
\end{example}
\begin{example}\label{175545_13Jan17}
 Let $\pi:C\rightarrow{E}$ be the bielliptic curve defined by
 $$
 C=\{[x:y:z]\in\mathbf{C}\mid
 (z^{2}+y^{2})^{2}=x(x-y)(x-5y)y\}.
 $$
 Then the $j$-function $\mathbf{j}_{(0,0,1,5)}$ in Lemma~\ref{162617_13Jan17} is computed by
 $$
 j_{F}(a)=\frac{2^{8}\cdot7(3a^{2}-3a+1)^{3}}
 {(a^{3}-6a^{2}+5a-1)^{2}},\quad
 j_{K}(a)=\frac{2^{4}(21a^{2}-30a+25)^{3}}
 {5^{2}(a(a-1)(a-5))^{2}}.
 $$
 Here we can check that $j_{K}$ and $j_{F}$ has no common ramification points on the $a$-line, and it is used in the proof of Lemma~\ref{162617_13Jan17}.
 Since
 $$
 \mathbf{j}_{(0,0,1,5)}(\infty)=(j_{F}(\infty),j_{K}(\infty))
 =(j(K_{Y(C/E)}),j(E))
 =\bigl(2^{8}\cdot3^{3}\cdot7,\
 \frac{2^{4}\cdot3^{3}\cdot7^{3}}{5^{2}}\bigr),
 $$
 we can check that
 $(j(E),j(K_{Y(C/E)}))=\bigl(\frac{2^{4}\cdot3^{3}\cdot7^{3}}{5^{2}},
 2^{8}\cdot3^{3}\cdot7\bigr)
 \in\Image{(\check{\mathbf{j}}_{(0,0,1,5)})}$
 is not contained in $\Image{(\mathbf{j}_{(0,0,1,5)})}$.
 Hence we have
 $\Image{(\mathbf{j}_{(0,0,1,5)})}\neq
 \Image{(\check{\mathbf{j}}_{(0,0,1,5)})}$,
 and it is used in the proof of Lemma~\ref{162638_13Jan17}.\par
 By Lemma~\ref{172705_10Dec16}, the bielliptic curve $C\rightarrow{E}$ defines a family $\{D_{a}\rightarrow{F_{a}}\}_{a}$ of bielliptic curves by
 \begin{multline*}
 D_{a}=\{[x:y:z]\in\mathbf{P}^{2}\mid\\
 z^{4}-2\frac{x^{2}+2axy-(2a^{2}-12a+5)y^{2}}{a^{3}-6a^{2}+5a-1}z^{2}
  -\frac{x^{4}+4ax^{3}y+2(12a-5)x^{2}y^{2}+20axy^{3}+25y^{4}}
  {a^{3}-6a^{2}+5a-1}=0\},
 \end{multline*}
 whose Prym variety $P=\Prym{(D_{a}/F_{a})}$ is the dual of the Prym variety $A=\Prym{(C/E)}$.
 The image of the $j$-function $\mathbf{j}_{(0,0,1,5)}:\mathbf{P}^{1}(a\text{-line})\rightarrow\mathbf{P}^{1}\times\mathbf{P}^{1}$ has many singular points, which give counter examples to the global Torelli theorem.
 For example, the point
 $$
 \mathbf{j}_{(0,0,1,5)}\bigl(-\frac{1}{3}\bigr)
 =\mathbf{j}_{(0,0,1,5)}\bigl(\frac{5}{6}\bigr)
 =\bigl(\frac{2^{8}\cdot3^{3}\cdot7^{2}}{13^{2}},\
 \frac{2^{4}\cdot3^{3}\cdot7^{3}}{5^{2}}\bigr)
 \in\mathbf{P}^{1}\times\mathbf{P}^{1}
 $$
 is a node of $\Image{(\mathbf{j}_{(0,0,1,5)})}$.
 This means that, elliptic surfaces
 $Y_{-\frac{1}{3}}$ and $Y_{\frac{5}{6}}$
 are not isomorphic each other, but their Hodge structures are isomorphic
 $$
 H^{2}(Y_{-\frac{1}{3}},\mathbf{Z})\simeq
 H^{2}(Y_{\frac{5}{6}},\mathbf{Z}),\quad
 H^{1}(Y_{-\frac{1}{3}},\mathbf{Z})\simeq
 H^{1}(Y_{\frac{5}{6}},\mathbf{Z}),\quad
 H^{1}(K_{Y_{-\frac{1}{3}}},\mathbf{Z})\simeq
 H^{1}(K_{Y_{\frac{5}{6}}},\mathbf{Z}),
 $$
 where we set $Y_{a}=Y(D_{a}/F_{a})$.
\end{example}

\begin{example}\label{165513_18Jan17}
 Let $\pi:C\rightarrow{E}$ be the bielliptic curve defined by
 $$
 C=\{[x:y:z]\in\mathbf{C}\mid
 (z^{2}+x(x-y))^{2}=x(x-y)(x+3y)y\}.
 $$
 Then the $j$-function $\mathbf{j}_{(1,-1,0,-3)}$ in Lemma~\ref{162617_13Jan17} is computed by
 $$
 j_{F}(a)
 =\frac{2^{6}(7a^{2}+18a+27)^{3}}
 {3^{2}((a+1)(a-3)(a+3))^{2}},\quad
 j_{K}(a)
 =\frac{2^{4}(13a^{2}-6a+9)^{3}}
 {3^{2}({a}({a-1})({a+3}))^{2}},
 $$
 and the family $\{D_{a}\rightarrow{F_{a}}\}_{a}$ of bielliptic curves is defined by
 \begin{multline*}
 D_{a}=\{[x:y:z]\in\mathbf{P}^{2}\mid\\
 z^{4}+2\frac{(2a-1)x^{2}+2axy-3y^{2}}{(a+1)(a-3)}z^{2}
  -\frac{x^{4}+4ax^{3}y-2(4a-3)x^{2}y^{2}-12axy^{3}+9y^{4}}
  {(a+1)(a-3)(a+3)}=0\}.
 \end{multline*}
 Then for $a=-1,9,-\frac{1}{3},\frac{3}{5},\frac{9}{5},\frac{3}{11}$,
 $$
 {F_{a}}\simeq{E},\quad
 \Prym{(D_{a}/F_{a})}\simeq\Prym{(C/E)}^{\vee},
 $$
 and for $a=\infty,-5,-\frac{3}{2},\frac{3}{5},-\frac{3}{7},-\frac{15}{7}$,
 $$
 {F_{a}^{\vee}}\simeq{E},\quad
 \Prym{(D_{a}^{\vee}/F_{a}^{\vee})}\simeq\Prym{(C/E)},
 $$
 where $D_{a}^{\vee}\rightarrow{F_{a}^{\vee}}$ denotes the dual of $D_{a}\rightarrow{F_{a}}$.
 The $12$ elliptic surfaces
 $$
 \begin{cases}
  Y_{a}=Y(D_{a}/F_{a})&
  (a=-1,9,-\frac{1}{3},\frac{3}{5},\frac{9}{5},\frac{3}{11}),\\
  Y_{a}^{\vee}=Y(D^{\vee}_{a}/F^{\vee}_{a})&
 (a=\infty,-5,-\frac{3}{2},\frac{3}{5},-\frac{3}{7},-\frac{15}{7})
 \end{cases}
 $$
 have the same Hodge structure as $Y=Y(C/E)$;
 $$
 \begin{cases}
  H^{1}(Y_{a},\mathbf{Z})\oplus{H^{2}(Y_{a},\mathbf{Z})}\simeq
  H^{1}(Y,\mathbf{Z})\oplus{H^{2}(Y,\mathbf{Z})}&
  (a=-1,9,-\frac{1}{3},\frac{3}{5},\frac{9}{5},\frac{3}{11}),\\
  H^{1}(Y_{a}^{\vee},\mathbf{Z})\oplus{H^{2}(Y_{a}^{\vee},\mathbf{Z})}\simeq
  H^{1}(Y,\mathbf{Z})\oplus{H^{2}(Y,\mathbf{Z})}&
  (a=\infty,-5,-\frac{3}{2},\frac{3}{5},-\frac{3}{7},-\frac{15}{7}),
 \end{cases}
 $$
 and they give the set
 $\mathcal{M}(H^{2}(Y)\oplus{H^{1}(Y)})$ of $12$ points in
 Theorem~\ref{163939_14Jan17}.
 The elliptic surface
 $[Y_{\infty}^{\vee}]\in\mathcal{M}(H^{2}(Y)\oplus{H^{1}(Y)})$
 is the only member which have the same canonical divisor
 as $Y$;
 ${H^{1}(K_{Y_{\infty}^{\vee}},\mathbf{Z})}\simeq
 {H^{1}(K_{Y},\mathbf{Z})}$.
 Since $C\rightarrow{E}$ is isomorphic to ${D^{\vee}_{\infty}\rightarrow{F^{\vee}_{\infty}}}$,
 the elliptic surface $Y=Y(C/E)$ is determined by the Hodge structure
 $H^{1}(K_{Y},\mathbf{Z})\oplus
 H^{1}(Y,\mathbf{Z})\oplus
 H^{2}(Y,\mathbf{Z})$.\par
 Next we gives an example corresponds to the point of
 $\Image{(\mathbf{j}_{(1,-1,0,-3)})}\cap
 \Image{(\check{\mathbf{j}}_{(1,-1,0,-3)})}$,
 by using this family.
 Since
 $$
 \mathbf{j}_{(1,-1,0,-3)}(-\frac{3}{2})
 =\bigl(\frac{2^{6}\cdot7^{3}}{3^{2}},\
 \frac{2^{4}\cdot3^{3}\cdot7^{3}}{5^{2}}\bigr),\quad
 \mathbf{j}_{(1,-1,0,-3)}(9)
 =\bigl(\frac{2^{4}\cdot3^{3}\cdot7^{3}}{5^{2}},\
 \frac{2^{6}\cdot7^{3}}{3^{2}}\bigr)
 $$
 the point
 $\bigl(\frac{2^{6}\cdot7^{3}}{3^{2}},
 \frac{2^{4}\cdot3^{3}\cdot7^{3}}{5^{2}}\bigr)$
 is contained in
 $\Image{(\mathbf{j}_{(1,-1,0,-3)})}\cap
 \Image{(\check{\mathbf{j}}_{(1,-1,0,-3)})}$.
 It means that $Y_{-\frac{3}{2}}$ and
 $Y_{9}^{\vee}$ has the same Hodge structure and same canonical divisor;
 $$
 H^{2}(Y_{-\frac{2}{3}},\mathbf{Z})\simeq
 H^{2}(Y_{9}^{\vee},\mathbf{Z}),\quad
 H^{1}(Y_{-\frac{2}{3}},\mathbf{Z})\simeq
 H^{1}(Y_{9}^{\vee},\mathbf{Z}),\quad
 H^{1}(K_{Y_{-\frac{2}{3}}},\mathbf{Z})\simeq
 H^{1}(K_{Y_{9}^{\vee}},\mathbf{Z}).
 $$
 But in this case, the Prym varieties are not isomorphic;
 $$
 \Prym{(D_{-\frac{3}{2}}/F_{-\frac{3}{2}})}
 \ncong\Prym{(D_{-\frac{3}{2}}^{\vee}/F_{-\frac{3}{2}}^{\vee})}
 =\Prym{(D_{9}^{\vee}/F_{9}^{\vee})}.
 $$
\end{example}

\begin{example}\label{165437_18Jan17}
 Let $\pi:C\rightarrow{E}$ be the bielliptic curve defined by
 $$
 C=\{[x:y:z]\in\mathbf{C}\mid
 (z^{2}+x^{2}-2xy+3y^{2}))^{2}=x(x-y)(x-3y)y\}.
 $$
 Then the $j$-function $\mathbf{j}_{(1,-2,3,3)}$ in Lemma~\ref{162617_13Jan17} is computed by
 $$
 j_{F}(a)
 =\frac{2^{6}(79a^{2}-84a+441)^{3}}
 {3^{2}\cdot5^{2}\cdot7^{2}(a(a-3)(a+7))^{2}},\quad
 j_{K}(a)
 =\frac{2^{6}(7a^{2}-12a+9)^{3}}{3^{2}(a(a-1)(a-3))^{2}},
 $$
 and the family $\{D_{a}\rightarrow{F_{a}}\}_{a}$ of bielliptic curves is defined by
 \begin{multline*}
 D_{a}=\{[x:y:z]\in\mathbf{P}^{2}\mid\\
 z^{4}+4\frac{x^{2}+2xy+3y^{2}}{a+7}z^{2}
  -\frac{x^{4}+4ax^{3}y+2(8a-3)x^{2}y^{2}+12axy^{3}+9y^{4}}
  {(a-3)(a+7)}=0\}.
 \end{multline*}
 Since $C\rightarrow{E}$ is isomorphic to $D_{\frac{7}{9}}\rightarrow{F_{\frac{7}{9}}}$, the original bielliptic curve is contained in this family.
 It means that the Prym variety $\Prym{(C/E)}$ is isomorphic to its dual
 $\Prym{(D_{a}/F_{a})}=\Prym{(C^{\vee}/E^{\vee})}$.
 We remark that $C\rightarrow{E}$ is not isomorphic to its dual $C^{\vee}\rightarrow{E^{\vee}}$.
 Using this family, we give an example of elliptic surface which has the same Hodge structure as its dual.
 Since
 $$
 \mathbf{j}_{(1,-2,3,3)}\bigl(-\frac{4}{3}\bigr)
 =\mathbf{j}_{(1,-2,3,3)}\bigl(\frac{21}{11}\bigr)
 =\bigl(\frac{2^{6}\cdot3^{3}\cdot13^3}{5^2\cdot7^{2}},\
 \frac{2^{6}\cdot3^{3}\cdot13^3}{5^2\cdot7^{2}}\bigr)
 $$
 is a node of
 $\Image{(\mathbf{j}_{(1,-2,3,3)})}=\Image{(\check{\mathbf{j}}_{(1,-2,3,3)})}$,
 the elliptic surfaces $Y_{-\frac{4}{3}}$ and $Y_{\frac{21}{11}}$
 are not isomorphic each other, but their Hodge structures are isomorphic
 $$
 H^{2}(Y_{-\frac{4}{3}},\mathbf{Z})\simeq
 H^{2}(Y_{\frac{21}{11}},\mathbf{Z}),\quad
 H^{1}(Y_{-\frac{4}{3}},\mathbf{Z})\simeq
 H^{1}(K_{Y_{-\frac{4}{3}}},\mathbf{Z})\simeq
 H^{1}(Y_{\frac{21}{11}},\mathbf{Z})\simeq
 H^{1}(K_{Y_{\frac{21}{11}}},\mathbf{Z}).
 $$
 We remark that the bielliptic curves $D_{-\frac{4}{3}}\rightarrow{F_{-\frac{4}{3}}}$ and
 $D_{\frac{21}{11}}\rightarrow{F_{\frac{21}{11}}}$
 are dual to each other, and by suitable change of the coordinate, $D_{-\frac{4}{3}}$ and $D_{\frac{21}{11}}$ are defined by the following equation;
 $$
 \begin{cases}
  D_{-\frac{4}{3}}\simeq
  \{[x:y:z]\in\mathbf{P}^{2}\mid
  (z^{2}+x^{2}-2xy+\frac{7}{5}y^{2})^{2}=-\frac{5}{3}x(x-y)(x-\frac{7}{5}y)y\},\\
  D_{\frac{21}{11}}\simeq
  \{[x:y:z]\in\mathbf{P}^{2}\mid
  (z^{2}+\frac{5}{7}x^{2}-2xy+y^{2})^{2}=-\frac{5}{3}x(x-y)(x-\frac{7}{5}y)y\}.
 \end{cases}
 $$
\end{example}

\subsection*{Acknowledgment}
The author would like to thank Juan Carlos Naranjo and Gian Pietro Pirola
for giving lectures on Prym varieties and Galois covering in Pragmatic 2016 at the University of Catania.
The idea of using the Prym varieties was motivated by their lectures.
This paper was prepared when the author stayed at the University of Pavia, and he is grateful for the hospitality.
He also thanks Masa-Hiko Saito, Sampei Usui and Kazuhiro Konno for responding to a question about the paper \cite{S}.

\bigskip
\begin{flushleft}
 {\sc Department of Mathematics, School of Engineering\\
 Tokyo Denki University, Adachi-ku\\
 Tokyo 120-8551, Japan}\\
 {\it E-mail address}:
 {\ttfamily atsushi@mail.dendai.ac.jp}
\end{flushleft}
\end{document}